\documentclass[reqno,a4paper,11pt]{amsart}
\usepackage{amsfonts,amsmath,amsthm,amssymb,amscd}
\usepackage{bbm}
\usepackage{comment}
\usepackage{color}
\usepackage{amsthm}
\usepackage{chngcntr}
\usepackage[mathscr]{euscript}
\usepackage{epsf}
\usepackage{graphicx,psfrag}
\usepackage{tabularx}
\usepackage{multicol}

\usepackage{frcursive}
\usepackage[T1]{fontenc}  

\usepackage{helvet}

\font\sc=cmcsc10 at 12truept 

\def \ScptB{{\mathcal B}}

\def \mF{F}

\def \mG{G}
\def \mM{M}
\def \mN{N}

\def \mycolor{\color{black}}

\def \myAd{\text{\rm{Ad}}}
\def \myFT{\text{\rm{FT}}}

\def \myId{\text{\rm{Id}}}
\def \myPV{\text{\rm{PV}}}

\def \myexp{\text{\rm{exp}}}
\def \mylog{\text{\rm{log}}}

\def \GL{\text{\rm{GL}}}
\def \Hom{\text{\rm{Hom}}}

\def \meas{\text{\rm{meas}}}
\def \SL{\text{\rm{SL}}}

\def \sgn{\text{\rm{sgn}}}

\def \mytrace{\text{\rm{trace}}}

\def \val{\text{\rm{val}}}
\def \myval{\text{\rm{val}}}

\def \myexp{\text{\rm{exp}}}
\def \mylog{\text{\rm{log}}}
\def \myG{\text{\rm{$G$}}}

\def \mytotal{\sigma}

\def \myAd{\text{\rm{Ad}}}

\def \mycInd#1#2{\text{\rm{c-Ind}}{\hskip 0.005in}^{#1}_{#2}}

\def \bF{{\mathbb F}}

\def \bN{{\mathbb N}}

\def \fkg{{\mathfrak g}}

\def \fkk{{\mathfrak k}}

\def \fkl{{\mathfrak l}}

\def \fks{{\mathfrak s}}

\def \myFF{{\mathbb F}_{q}}
\def \myFFsl2{{\mathfrak s}{\mathfrak l}  (2,{\mathbb F}_{q})}
\def \myFFGL2{{\text{\rm{GL}}}  (2,{\mathbb F}_{q})}
\def \myFFSL2{{\text{\rm{SL}}}  (2,{\mathbb F}_{q})}

\def \myauthor{Allen Moy}

\def\dsize{\displaystyle}

\catcode`\@=11

\@addtoreset{equation}{section} \@addtoreset{equation}{subsection}
\def\theequation{\ifnum\value{subsection}>0\relax
\thesubsection.\arabic{equation}\relax
\else\ifnum\value{section}>0\relax
\thesection.\arabic{equation}\relax \else\arabic{equation}\fi\fi}

\newtheorem{thm}[equation]{Theorem}
\newtheorem{lemma}[equation]{Lemma}
\newtheorem{prop}[equation]{Proposition}

\newtheorem*{thm*}{Theorem}
\newtheorem*{prop*}{Proposition}

\newtheorem{cor}[equation]{Corollary}

\newcommand \reBna{B}

\newcommand \reBD{BD}
\newcommand \reBKV{BKV}

\newcommand \reDa{D}
\newcommand \reIR{IR}
\newcommand \reHCa{HC}
\newcommand \reKa{Ka}
\newcommand \reKb{Kb}

\newcommand \reMPa{MPa}
\newcommand \reMPb{MPb}
\newcommand \reMTa{MT1}
\newcommand \reMTb{MTb}

\newcommand \reSSa{SS}

\newcommand{\xRightarrow}[2][]{\ext@arrow 0359\Rightarrowfill@{#1}{#2}}

\begin{document}

{\large{


\vskip 0.30in

\noindent{\hfill}{In honor of Roger Howe as a septuagenarian, and}

\smallskip

\noindent{\hfill}{in memory of Paul Sally Jr.~and Joseph Shalika}

\smallskip

\vskip 0.30in

\title[Computations with Bernstein projectors of SL(2)]  
{Computations with Bernstein projectors of SL(2)}

\markboth{\myauthor}{Bernstein projectors}

\author{\myauthor}

\address
{Department of Mathematics, The Hong Kong University of Science and Technology \\
Clear Water Bay, Hong Kong \\
Email:\tt amoy{\char'100}ust.hk}

\address{Department of Mathematics, University of Utah, Salt Lake City, UT 84112}

\thanks{The author is partly supported by Hong Kong Research Grants Council
grant CERG {\#}603813.}



\maketitle

\begin{abstract}

For the p-adic group $G=\SL (2)$, we present results of the computations of the sums of the Bernstein projectors of a given depth. Motivation for the computations is based on a conversation with Roger Howe in August 2013. The computations are elementary, but they provide an expansion of the delta distribution $\delta_{1_{G}}$ into an infinite sum of $G$-invariant locally integrable essentially compact distributions supported on the set of topologically unipotent elements.  When these distributions are transferred, by the exponential map, to the Lie algebra, they give $G$-invariant distributions supported on the set of topologically nilpotent elements, whose Fourier Transforms turn out to be characteristic functions of very natural $G$-domains. The computations in particular rely on the $\SL (2)$ discrete series character tables computed by Sally-Shalika in 1968. This new phenomenon for general rank has also been independently noticed in recent work of Bezrukavnikov, Kazhdan, and Varshavsky.

\end{abstract}

\vskip 0.30in

\section{Introduction}

\medskip

A key tool in harmonic analysis on a Lie group $\mG$ is the exponential map $\exp \, : \, \fkg \ \longrightarrow \ \mG$ from the Lie algebra $\fkg$ to the group $\mG$.  The map, defined for all $X \, \in \, \fkg$ is a local diffeomorphism of $0 \, \in \, \fkg$ to $1 \, \in \, \mG$, and it is used to move functions and distributions, between $\fkg$ and $\mG$, e.g., ones which are $\mG$-invariant and eigendistributions for the center of ${\mathcal U}(\fkg )$.   \ When $\mF$ is a p--adic field, i.e., a local field of characteristic zero, and $\mG$ the $\mF$-rational points of a connected reductive group defined over $\mF$, and $\fkg$ the Lie algebra of $\mG$, the exponential map is only defined on a certain $\mG$-invariant, open and closed subset containing $0 \in \fkg$.  In terms of Moy--Prasad filtrations [{\reMPa},{\reMPb}], for $r \in {\mathbb R}$, set
$$
\aligned
\fkg_{r} \, :&= \, {\underset {x \, \in \, {\text{\tiny{\fontfamily{frc}\selectfont{B}}}}} \bigcup } \fkg_{x,r} \quad {\text{\rm{and}}} \quad \fkg_{r^{+}} \, :&= \, {\underset {x \, \in \, {\text{\tiny{\fontfamily{frc}\selectfont{B}}}} \, , \, s > r} \bigcup } \fkg_{x,s} \ .
\endaligned
$$

\noindent{We} recall these sets are $\mG$-domains, i.e.,  $\mG$-invariant open and closed subsets.  The set $\fkg_{0^{+}}$ is the set of topologically nilpotent elements in $\fkg$.  The assumption char$(\mF) \, = \, 0$ means, there exists $R \, \ge \, 0$ so that for $r \, > \, R$, the exponential map exp is defined on the $G$-domain $\fkg_{r}$ and for any $x \in {\text{\tiny{\fontfamily{frc}\selectfont{B}}}}$, exp takes $\fkg_{x,r}$ bijectively to the Moy--Prasad group $G_{x,r}$, and $\fkg_{r}$ bijectively to $\mG_{r}$.  In the best situation $R$ is $0$.

\medskip

We recall the two realizations [{\reBna}] of the Bernstein center ${\mathcal Z}(\mG )$:

\smallskip

\begin{itemize}
\item[$\bullet$] \ The geometrical realization. \ ${\mathcal Z}(\mG )$ is the algebra of $\mG$-invariant  essentially compact distributions on $\mG$ -- a distribution is essentially compact if $\forall \, f \in C^{\infty}_{0}(\mG )$, the convolutions $D \star f$ and $f \star D$ is in $C^{\infty}_{0}(\mG )$.

\smallskip

\item[$\bullet$] \ The spectral realization. \ ${\mathcal Z}(\mG )$ is the product \ ${\underset {\Omega} \prod } \ {\mathbb C}(\Omega )$ \ over the Bernstein components $\Omega$ of the algebras of (complex) regular functions ${\mathbb C}(\Omega )$.  

\end{itemize}

\smallskip

\noindent{For} a fixed Bernstein component $\Omega$, and a regular function $s \, \in \, {\mathbb C}(\Omega )$, it is known the distribution $s$ is representable by a locally integrable function supported, and locally constant on the regular set.  Of particular interest is the idempotent distribution $e_{\Omega}$ whose spectral realization is the constant function 1 on $\Omega$.  In the situation when $\mG$ is semisimple and the Bernstein $\Omega$ corresponds to an equivalence class of supercuspidal representations, then $e_{\Omega}$ is the distributional character of the class times its formal degree.   In this setting ($\mG$ semisimple), an important result of Dat [{\reDa}] states the distribution $e_{\Omega}$ is supported on the $\mG$-domain ${\mathcal C}$  of compact elements, i.e., elements which belong to a  compact subgroup of $\mG$.   We note the set of topologically unipotent elements ${\mathcal U}^{\text{\rm{top}}}$ is contained in ${\mathcal C}$.

\smallskip

Based on known examples, e.g., for $\SL(2)$ see [{\reMTa},{\reMTb}], if $s \in {\mathbb C}(\Omega )$ is viewed as a distribution, its support is generally not contained in ${\mathcal U}^{\text{\rm{top}}}$.  An intriguing question is whether it is possible to find an element in the span of finitely many ${\mathbb C}(\Omega )$ whose support is contained in ${\mathcal U}^{\text{\rm{top}}}$. 

\smallskip

In [{\reMPa,\reMPb}], the depth invariant $\rho (\pi)$ is defined for any irreducible smooth representation $\pi$, and it is known the depth is the same for all irreducible representation (classes) occurring in a Bernstein component.  Thus, a depth $\rho (\Omega )$ is attached to any Bernstein component $\Omega$.

\bigskip

Let $d \, \ge \, 0$ be the depth of a Bernstein component, and set 
$$
\aligned
e_{d} \ :&= \ {\underset {\rho (\Omega ) \, = \, d} \sum } \ e_{\Omega} \quad {\text{\rm{and}}} \quad 
\sigma_{d} \ := \ {\underset {\rho (\Omega ) \, \le \, d} \sum } \ e_{\Omega} \ . \\ 
\endaligned
$$

Here we show, for $\mG \, = \, \SL (2, \mF )$, the Bernstein center element $\sigma_{d}$ has support in the topological unipotent set 
$$
{\mathcal U}^{\text{\rm{top}}}_{d^{+}} \, = \, {\underset {x \, \in \, {\text{\tiny{\fontfamily{frc}\selectfont{B}}}}} \bigcup } \mG_{x,d^{+}}
$$ 

\noindent{Recall} \  (i) an element $y \in \mG$ is called split (resp.~elliptic), if its characteristic polynomial has distinct roots in $\mF$ (resp.~not in $\mF$), and (ii) the depth of an irreducible representations is a half-integer, i.e, in ${\frac12}{\mathbb N}$.


\begin{thm*}{\bf \ref{main-result-group}}{\bf{.}} \ \ Suppose $\mF$ is a p--adic field with odd residue characteristic, and $\mG \, = \, \SL (2,\mF )$.    For $d \, \in \, {\frac12}{\mathbb N}$, set $d^{+} \, := \, k \, + \, {\frac12}$.  \ Then, we have \ ${\text{\rm{supp}}} \, ( \,  {\mytotal}_{d} \, )  \ \subset \ {\mathcal U}^{\text{\rm{top}}}_{d^{+}}$, and on \ ${\mathcal U}^{\text{\rm{top}}}_{d^{+}}$:

\medskip

\begin{itemize}


{\mycolor

\item[$\bullet$] \ \ When  $d$ is integral:

$$
\mytotal_{d} \, ( \, y \, ) \, = \, 
{\text{\rm{\large $(q^2-1) \, q^{3d} $}}} \ 
\begin{cases}
\  \big( {\frac{2 \, q^{-d}}{| \, \alpha \, - \, \alpha^{-1} \, |_{\mF}}} \, - \, 1 \big)  \quad {\text{\rm{when $y$ is split with eigenvalues $\alpha$, $\alpha^{-1}$}}} \\ 
\ &\ \\
\ -1  \qquad {\text{\rm{when $y$ is elliptic}}} \\
\end{cases}
$$

\medskip

\item[$\bullet$] \ \ When $d$ is half-integral:

$$
\mytotal_{d} \, ( \, y \, ) \, = \, 
{\text{\rm{\large $(q^2-1) \, q^{3d+{\frac12}} $}}} \ 
\begin{cases}
\  \big( {\frac{2 \, q^{-(d+{\frac12})}}{| \, \alpha \, - \, \alpha^{-1} \, |_{\mF}}} \, - \, 1 \big) \qquad {\text{\rm{when $y$ is split with eigenvalues $\alpha$, $\alpha^{-1}$}}} \\
\ &\ \\
\ -1  \qquad {\text{\rm{when $y$ is elliptic}}} \\
\end{cases}
$$
}

\end{itemize}
\end{thm*}

We observe, for $\SL (2)$, the projector $\sigma_{0}$ is equal to the Steinberg character restricted to the topological unipotent set.  

\bigskip

As mentioned above, there is an $R \, \ge \, 0$ so that the exponential power series $\exp (X)$ is convergent when the eigenvalues of $X$ have normalize valuations greater than $R$, and for $r \, > \, R$, the exponential map is then a bijection between the set ${\mathcal N}^{\text{\rm{top}}}_{r} \, \subset \, \fkg$, and the set ${\mathcal U}^{\text{\rm{top}}}_{r} \, \subset \, \mG$.  In this ideal situation, we can then move the distributions in \eqref{main-result-group} to the Lie algebra.  For $d \, \in \, {\frac12}{\mathbb N}$, satisfying $d > R$, the Lie algebra distribution $\mytotal_{d}  \, \circ \, \myexp$ has support in $\fkg_{d^{+}}$, and we have the homogeneity relation 

\begin{equation}\label{homogeneity-relation-1}
( \, \mytotal_{d+1}  \, \circ \, \myexp \, ) \, ( \, \varpi Y \, ) \  = \  q^{3} \,  ( \, \mytotal_{d}  \, \circ \, \myexp \, ) \, ( \, Y \, ) \ .
\end{equation}

\bigskip

\noindent{Whence}, their Fourier transforms satisfy the homogeneity relation

\begin{equation}\label{homogeneity-relation-2}
\myFT \, ( \, \mytotal_{k+1}  \, \circ \, \myexp \, ) \, ( \, \varpi^{-1} Y \, ) \  = \ \myFT \,  ( \, \mytotal_{k}  \, \circ \, \myexp \, ) \, ( \, Y \, ) \ .
\end{equation}

\bigskip

\noindent{In} this regard, we show in the appendix the following:


\begin{prop*}{\bf \ref{appendix-prop}}{\bf{.}}   \ \ For $\fks \fkl (2)$, we have

\vskip 0.10in

\begin{itemize} 
\item[$\bullet$] The Fourier transforms $FT(1_{\fkg_{0}} )$ and $FT(1_{\fkg_{-{\frac12}}} )$ have support in the sets $\fkg_{0^{+}} := \fkg_{\frac12}$  and 
$\fkg_{ ({\frac12})^{+}} := \fkg_{1}$ respectively.  {\hfill} 

\noindent In particular, the support is contained in ${\mathcal N}^{\text{\rm{top}}}$. 

\vskip 0.10in

\item[$\bullet$] For $k \ge 1$, the Fourier transform
$FT(1_{\fkg_{-k}} )$ has support in $\fkg_{k^{+}} := \fkg_{k+\frac12}$.
\end{itemize}
\end{prop*}

For a general connected reductive p-adic group, when $R=0$ and other conditions, Kim [{\reKa},{\reKb}], showed, for $X$ in $\fkg_{\text{\normalsize{\rm{$({\frac{d}{2}})^{+}$}}}}$:

$$
\int_{\widehat{G}^{{\text{\tiny{\rm{{\,}temp}}}}}_{\le d}} \ \Theta_{\pi} ( \, \exp (X) \, ) \, d \mu_{_{\text{PM}}} \, ( \pi ) \ = \ 
FT(1_{\fkg_{-d}}) \, (X) \ .
$$

\noindent{The} integral is over the (classes of) irreducible tempered representations of depth less than or equal to $d$.  Thus, for $\SL (2)$, when $R \, = \, 0$ we have 
$$
\sigma_{d} \circ \exp  \ = \ FT(1_{\fkg_{-d}}) \qquad {\text{\rm{(both sides have support in $\fkg_{d^{+}}$).}}}
$$

\noindent{We} conjecture, for $\SL(2)$, and more generally for any connected reductive p-adic group, this identity to be true even when $R \, > \, 0$, as long as $d > R$. 


\bigskip

In October 2014, through correspondence with Roman Bezrukavnikov, the author became aware of unpublished work in-progress of Bezrukavnikov, Kazhdan, and Varshavsky in which they independently discovered and proved, for a general connected reductive p--adic group, the support of the projector $\sigma_{d}$ is in the topological unipotent set, that $\sigma_{0}$ is the restriction of the Steinberg character to the unipotent set, and the identification of the Fourier transforms $\myFT( \, \sigma_{d} \circ \exp \, )$.  A preprint [{\reBKV}] of their work became available in April 2015.

\vskip 0.10in

Motivation for considering the support of the depth zero projector 
$e_0$, the sum of the projectors $e_{\Omega}$ with $\rho (\Omega ) \, = \, 0$, aroused during a conversation the author had with Roger Howe in August 2013, and the author successfully verified the support is contained in ${\mathcal U}^{\text{\rm{top}}}_{0^{+}}$ and a formula for the values in December 2013.  Extension of the support and values of $\sigma_{d}$ to all depths was completed in March 2015 while the author was a visiting faculty at the University of Utah.  The author kindly thanks the hospitality of the Mathematics Department of the University of Utah, with special thanks to Dragan Mili{\v{c}}i{\'{c}}.  The author gratefully acknowledges useful conversations with Roman Bezrukavnikov, Roger Howe, 
Ju-Lee Kim, and Fiona Murnaghan.   The author gave workshop talks of the case $e_0$ at the Mathematical Research Institute of Oberwolfach, the University of Zagreb, and the University of Chicago, and thanks these institutions for their invitations.

\vskip 1.00in

\section{Notation}

\medskip

We set some notation.  \ Let $\mF$ denote a p-adic field (so of characteristic zero) .  Let ${\mathcal R}_{\mF}$ denote the ring of integers of $\mF$, let $\wp_{\mF}$ its prime ideal, and let $\varpi$ be a prime element.  Set $\bF_{q} \, = \, {\mathcal R}_{\mF}/\wp_{\mF}$ to be the residue field.  To be able to use the Sally-Shalika character tables [{\reSSa}], we assume the residue characteristic of $\mF$ is odd.

\medskip 

 Let $\ScptB$ be the Bruhat-Tits building $\myG \, = \, \SL (2, \mF )$.  The group $\GL (2, \mF )$, whence also $\myG$, acts on $\ScptB$. There are respectively two $\myG$-orbits, and one $\GL (2,\mF )$-orbit of vertices in $\ScptB$.   The maximal compact subgroups of $\myG$ are precisely the stabilizer subgroups $G_{x}$ (in $\myG$) of vertices $x$ in $\ScptB$; whence, there two conjugacy classes of maximal compact subgroups in $\myG$.   Let $x_0$ and $x_1$ be the vertices in $\ScptB$ so that 

\begin{equation}\label{x0x1}
\aligned
G_{x_0} \ = \ \SL (2, {\mathcal R}_{\mF}) \ \ , \quad  &{\text{\rm{and}}} \quad G_{x_1} \ = \  \left[ \begin{matrix} \varpi^{-1} &{\ 0 \ } \\ 0 & 1 \end{matrix} \right] \, \SL (2, {\mathcal R}_{\mF}) \, \left[ \begin{matrix} \varpi &{\ 0 \ } \\ 0 & 1 \end{matrix} \right] \\
\endaligned
\end{equation}

\medskip

\noindent are the familiar representatives of the two conjugacy classes of maximal compact subgroups of $\myG$. \ \ If $e$ is an edge (with vertices $y$ and $z$ and (midpoint) barycenter $b(e)$) in $\ScptB$, then the stabilizer (in $\myG$) of $e$ is an Iwahori subgroup equal to $G_{y} \cap G_{z} \ = \ G_{b(e)}$.

\medskip 

We note the two vertices $x_0$ and $x_1$ mentioned in \eqref{x0x1} are the vertices of an edge $e_{01} \in \ScptB$.  Let $x_{01} \, = \, b(e_{01})$ -- the barycenter of $e_{01}$.  \ The Iwahori subgroup $\mG_{x_{01}}$ equals:

\begin{equation}\label{iwahori-1}
G_{x_{01}} \ = \ G_{x_0} \ \cap \ G_{x_1} \ = \ \{ \ g \in \SL (2 , {\mathcal R}_{\mF}) \ | \ {\text{\rm{$g$ upper triangular mod $\wp_{\mF}$}}} \ \} \ . 
\end{equation}

\noindent{For} notational convenience, we set 

$$
{\mathcal K} \, = \, G_{x_0} \ \ , \quad  {\mathcal K}' \, = \, G_{x_1} \ \ , 
\quad  {\mathcal I} \, = \, {\mathcal K} \, \cap {\mathcal K}' \, = \, G_{x_{01}}
\ .
$$

\bigskip

Set $\fkg  \, = \, \fks \fkl (2, \mF)$.  For any $x \in \ScptB$, let  
$$
\myG_{x,r} \ \ {\text{\rm{and $r \, \ge \, 0$}}} \ \ \ {\text{\rm{and}}} \ \ \ \fkg_{x,r} \ \ {\text{\rm{and $r \in {\mathbb R}$}}}
$$

\noindent{be} the Moy--Prasad filtration subgroups. \ We recall: 

\medskip

\begin{itemize} 
\item[(i)] If $x \in \ScptB$ is a vertex, then the jumps in the filtration subgroups $G_{x,r}$ occur at integral $r$, i.e., 
$$
G_{x,r^{+}} \, = \, G_{x,r+1} \quad {\text{\rm{when}}} \quad r \in {\mathbb N}  
$$
\item[(ii)]  If $x \, = \, b(e)$ is the barycenter of an edge $e$, then the jumps in the filtration subgroups occur at half-integral $r$, i.e., 
$$
G_{x,r^{+}} \, = \, G_{x,r+{\frac12}} \quad {\text{\rm{when}}}  \quad r \, \in \, {\frac12} {\mathbb N} \ ,
$$
\end{itemize}
\medskip
\noindent{and} similarly for $\fkg_{x,r}$.  For the latter, we have 
$$
\fkg_{x,r+1} \ = \ \varpi \fkg_{x,r} \ . 
$$ 
\medskip

\noindent Take $\psi$ to be an additive character of $\mF$ which has conductor $\wp_{\mF}$. 

\medskip
 
\begin{itemize} 

\item[(iii)] For $2r \, \ge \, s$, the quotient group $G_{x,r}/G_{x,s}$ is abelian and canonically isomorphic to $\fkg_{x,r}/\fkg_{x,s}$.  The residual characteristic is odd assumption means the trace pairing  
$$
\fkg_{x,r}/\fkg_{x,s} \ \times \ \fkg_{x,-s^{+}}/\fkg_{x,-r^{+}} \ \rightarrow \ {\mF}/{\wp}_{\mF}
$$
allows us to identify the Pontryagin dual $(\fkg_{x,r}/\fkg_{x,s})^{\widehat{\ }}$ (whence $(\mG_{x,r}/\myG_{x,s})^{\widehat{\ }}$ with $\fkg_{x,-s^{+}}/\fkg_{x,-r^{+}}$.  A coset $\Xi \, = \, X \, + \, \fkg_{x,-r^{+}}$ yields the character $\psi_{\Xi}$ of 
$\fkg_{x,r}/\fkg_{x,s}$ given as:
$$
\psi_{\Xi} (Y) \ := \ \psi ( \, \mytrace (XY) \, ) \ .
$$

\noindent{We note}:  
\begin{itemize} 
\item[$\bullet$] For the vertex $x_0$: 
$$
\fkg_{x_{0},0}/\fkg_{x_{0},1} \ = \ \myFFsl2 \ .
$$
For a general vertex $x \in \ScptB$ a vertex, and $r \in {\mathbb N}$:
$$
(\fkg_{x,r}/\fkg_{x,(r+1)})^{\widehat{\ }} \ = \ \fkg_{x,-r}/\fkg_{x,(-r+1)} \ \simeq \ \fkg_{x,0}/\fkg_{x,1} \ \simeq \ \myFFsl2 
$$
The isomorphism $\fkg_{x,-r}/\fkg_{x,(-r+1)} \ \simeq \ \fkg_{x,0}/\fkg_{x,1}$ is the natural one given by scalar multiplication by $\varpi$.  The isomorphism $\fkg_{x,0}/\fkg_{x,1} \ \simeq \ \myFFsl2$ is up to a conjugation by by an element of $\myFFGL2$.  A coset 
$X \in \fkg_{x,-r}/\fkg_{x,(-r+1)}$ is, by definition, non-degenerate if, as an element in $\myFFsl2)$, it is non-nilpotent.  

\medskip

\item[$\bullet$] For the barycenter point $x_{01}$:
$$
\fkg_{x_{01},{\frac12}}/\fkg_{x_{01},1} \ = \ \big\{ \ \left[ \begin{matrix} \varpi \, a &b \\ \varpi \, c &-\varpi \, a \end{matrix} \right] \ | \ a, \, b, \, c \ \in \ {\mathcal R}_{\mF} \ \big\} \ / \ \big\{ \ \left[ \begin{matrix} \varpi \, a &\varpi \, b \\ \varpi^2 \, c &-\varpi \, a \end{matrix} \right] \ | \ a, \, b, \, c \ \in \ {\mathcal R}_{\mF} \ \big\} \ .
$$
\noindent{For} $d \in {\mathbb N}$:
$$
(\fkg_{x_{01},d+\frac12}/\fkg_{x_{01},(d+1)})^{\widehat{\ }} \ = \ \fkg_{x_{01},-d-{\frac12}}/\fkg_{x_{01},-d} \\
$$

\noindent{We} recall a coset 
\begin{equation}
X \ = \ \varpi^{-(d+1)} \left[ \begin{matrix} \varpi \, a &b \\ \varpi \, c &-\varpi \, a \end{matrix} \right] \ + \ \fkg_{x_{01},-d} \ \in \ \fkg_{x_{01},-d-{\frac12}}/\fkg_{x_{01},-d} \quad (a, \, b, \, c \ \in \ {\mathcal R}_{\mF})  \ ,
\end{equation}
\noindent{is} non-degenerate if $b$ and $c$ are both units.
\end{itemize}
\end{itemize}

\vskip 1.0in 


\section{Review of earlier results}

\medskip

\subsection{A result of J.~Dat on support} \ \ Suppose $\mG$ is a general connected reductive p--adic group.  A {\it compact element} $\gamma \in G$ is one which lies in a compact subgroup of $G$.  Set
$$
{\mathcal C} \ := \ {\text{\rm{set of compact elements.}}} 
$$

\noindent{The} following important result of Dat says the support of a projector $e_{\Omega}$ to a Bernstein component $\Omega$ is contained in ${\mathcal C}$. 

\smallskip

\begin{thm} \ {\text{\rm{(Dat, [{\reDa}])}}} \ Suppose  $\mG$ is a connected reductive p--adic group.  Let ${\mathcal C}$ be the set of compact elements in $\mG$.  For any Bernstein component $\Omega$, let $e_{\Omega}$ be the element of the Bernstein center which projects onto the component $\Omega$.  Then 
$$
{\text{\rm{supp}}}{\,}(e_{\Omega}) \ \subset \ {\mathcal C} \ .
$$
\end{thm}

\bigskip

\subsection{A result of Moy--Prasad on depths} \ \  If $\pi$ is an irreducible smooth representation of a connected reductive p--adic group, let $\rho (\pi )$ denote its depth as defined in [{\reMPa},{\reMPb}].  We recall the following result on depths of representations which implies we can define the depth of a Bernstein component.

\begin{thm} \ {\text{\rm{(Moy--Prasad, [{\reMPb}], Theorem 5.2)}}} \ Suppose  $\mG$ is a connected reductive p--adic group, $\mM \mN$ is a parabolic subgroup, and $\sigma$ an irreducible smooth representation of $\mM$.   If $\pi$ is any irreducible subquotient of ${\text{\rm{Ind}}}^{\mG}_{\mM  \mN} \, ( \sigma )$, then $\rho (\pi) \, = \, \rho (\sigma)$.
\end{thm}

\noindent{Let} $d \, \ge \, 0$ be the depth of a Bernstein component, and set 
$$
\aligned
e_{d} \ :&= \ {\underset {\rho (\Omega ) \, = \, d} \sum } \ e_{\Omega} \quad {\text{\rm{and}}} \quad 
\sigma_{d} \ := \ {\underset {\rho (\Omega ) \, \le \, d} \sum } \ e_{\Omega} \ . \\ 
\endaligned
$$

As mentioned in the introduction, when $\mG \, = \, \SL (2)$, in the remainder of this manuscript, we show $\sigma_{d}$ has support in the topological unipotent set ${\mathcal U}^{\text{\rm{top}}}$, and indeed in the smaller $\mG$-domain ${\mathcal U}^{\text{\rm{top}}}_{d^{+}}$.

\bigskip
\bigskip

\subsection{A partition of the compact elements ${\mathcal C}$ of $\SL(2)$}  \ \  We partition ${\mathcal C}$ into three subsets:

$$
{\mathcal C} \ \ = \ \ {\mathcal U}^{\text{\rm{top}}} \ \ \coprod \ \  
-{\text{\rm{I}}}_{2 \times 2} \, {\mathcal U}^{\text{\rm{top}}} \ \ \coprod \ \  
{\mathcal C}_{\text{\rm{st-reg}}}
$$

\noindent Here, ${\mathcal C}_{\text{\rm{st-reg}}}$ is the set of `strongly regular' elements, i.e., those elements whose eigenvalues are not congruent to each other modulo the prime.

\vskip 1.0in

\section{Principal Series projectors for $\SL(2)$}

\medskip

\noindent For $\SL(2)$, Moy--Tadi{\'{c}} [{\reMTa},{\reMTb}] explicitly computed  the projectors $e_{\Omega}$ for principal series components.  To state the results, we normalize Haar measure on $\SL (2)$ so that
\begin{equation}\label{haar-measure-normalization}
\meas ( \, \SL (2, {\mathcal R}_{\mF} ) \, ) \ = \ 1 \ .
\end{equation}

\noindent{We} enumeration the Principal Series Bernstein components as: \ (i)  
$\Omega (\{ \chi , \chi^{-1} \})$, where $\chi$ is a character of 
${\mathcal R}^{\times}_{\mF}$ with $\chi \, \neq \, \chi^{-1}$, \ (ii) $\Omega (\sgn )$, with $\sgn$ the order two character of ${\mathcal R}^{\times}_{\mF}$, and \ (iii)  $\Omega \, = \, \Omega_{\text{\rm{triv}}}$, the Bernstein component of irreducible representations with non-zero Iwahori-fixed vectors. \ Let $f(\chi )$ denote the conductor of $\chi$.  \ The Principal Series projectors are:

$$
\aligned
{\text{\rm{{\bf Regular PS}  \ \ $f(\chi ) \, = \, d+1$  }}} {\hskip 0.30in} \ & \\  
e_{\Omega (\{ \chi , \chi^{-1} \})} ( \ y \ )
\ &= \ \begin{cases} 
\ \ (q+1) \ q^d \ \dsize {\frac{\chi (\alpha ) + \chi (\alpha^{-1} )}{ | \, \alpha - \alpha^{-1} \, |_{F} }} \ , \\ 
\ \qquad {\text{\rm{$y$ split with eigenvalues $\alpha$, $\alpha^{-1}$}}} \\
\ \ 0 \quad {\text{\rm{ otherwise}}}
\end{cases}
&\ \\
{\text{\rm{\bf Sgn PS}}} {\hskip 1.8in} \ & \\  
e_{\Omega (\sgn )} ( \ y \ )
\ &= \ \begin{cases} 
\ \ (q+1) \ \dsize {\frac{ \sgn ( \, \alpha \, )}{ | \, \alpha - \alpha^{-1} \, |_{F} }} \ \ , \\
\ \qquad {\text{\rm{$y$ split with eigenvalues $\alpha$, $\alpha^{-1}$}}} \\
\ \ 0 \quad {\text{\rm{ otherwise}}}
\end{cases}
&\ \\
&\ \\
{\text{\rm{\bf Unramified PS (Iwahori fixed vectors)}}} {\hskip -0.7in} \ & \\
e_{\Omega } ( \ y \ )
\ &= \ \begin{cases} 
\ \ \dsize {\frac{ 2 \, q }{ | \, \alpha - \alpha^{-1} \, |_{F} }} \ - \ (q-1) \ \ , \\
\ \qquad {\text{\rm{$y$ split with eigenvalues $\alpha$, $\alpha^{-1}$}}} \\
\ - (q-1) \quad {\text{\rm{ $y$ elliptic }}}  \\
\end{cases}
\endaligned
$$

\medskip

\noindent{Let} $e^{\text{PS}}_{d}$ be the sum of the principal series depth $d$ Bernstein projectors.  \ \  

\eject

\noindent{For} $d \, = \, 0$ we have:


{\mycolor

\begin{equation}\label{ps-zero}
e^{\text{PS}}_{0} \, ( y ) \ = \ (q-1) \ 
\begin{cases}
\ \ {\frac{(q+2)}{| \alpha - \alpha^{-1} |}} \ - \ 1  \\
{\hskip 0.50in} {\text{\rm{when \ $y \in {\mathcal U}^{\text{\rm{top}}}$ \ is split with eigenvalues $\alpha$, $\alpha^{-1}$ }}} \\
\ \ {\frac{1}{| \alpha - \alpha^{-1} |}} \ - \ 1  \\
{\hskip 0.50in} {\text{\rm{when \ $y \in -{\text{\rm{I}}}_{2 \times 2} \, {\mathcal U}^{\text{\rm{top}}}$ \ is split with eigenvalues $\alpha$, $\alpha^{-1}$ }}} \\
{\hskip 0.22in} 0  {\hskip 0.20in} {\text{\rm{when \ $y \in {\mathcal C}_{\text{\rm{st-reg}}}$ \ is split }}} \\
{\hskip 0.05in} -1  {\hskip 0.18in} {\text{\rm{when \ $y \in {\mathcal C}$ \ is elliptic }}} \\
\end{cases}
\end{equation}

}

\bigskip

\noindent For $d \, > \, 0$, each principal series Bernstein projector $e_{\Omega}$ has support in the split set and consequently the support of $e^{\text{PS}}_{d}$ is in the split set too.  For $y \in \SL (2)$ a (regular) split element, let $\alpha$, $\alpha^{-1}$ be the eigenvalues of $y$.  \ We have:
 

{\mycolor

\begin{equation}\label{ps-greater-zero}
e^{\text{PS}}_{d} \, ( y ) \ = \ (q+1) \, q^d \, (q-1) \, q^{d-1} \ 
\begin{cases}
\ \ {\frac{(-1)}{| \alpha - \alpha^{-1} |}} \ = \ - \, q^{d} \quad  {\text{\rm{when $\alpha$, $\alpha^{-1}$ in $(1+\wp^{d}_{\mF}) \backslash (1+\wp^{d+1}_{\mF}) $}}} \\
\ \\
\ \ {\frac{(q-1)}{| \alpha - \alpha^{-1} |}} \quad  {\text{\rm{when $\alpha$, $\alpha^{-1}$ in $(1+\wp^{d+1}_{\mF}) $}}} \\
\ \\
\ \ \ \ 0 {\hskip 0.47in}   {\text{\rm{otherwise}}} \\ 
\end{cases}
\end{equation}
}

\vskip 0.50in


\section{The projector $e_{0}$}

Set 

\begin{equation}\label{e-cusp-0-defintion}
e^{\text{\rm{cusp}}}_{0} \ = \  {\underset  {\begin{matrix} {\text{\rm{\small{$\pi$ cuspidal}}}} \\ {\text{\rm{\small{depth 0}}}} \end{matrix}} \sum } e_{\pi}  \ . 
\end{equation}
 
\noindent{In} this section, we determine, with the aid of the Sally-Shalika tables [{\reSSa}], the values of $e^{\text{\rm{cusp}}}_{0}$.  To conveniently use their tables we use their normalization of Haar measure, i.e., $\meas (\mG_{x_{0}} ) \, = \, 1$, as in \eqref{haar-measure-normalization}.

\medskip

\subsection{Irreducible cuspidal representations of depth zero} \quad Let ${\mathcal K} \, = \, G_{x_{0}}$, and ${\mathcal K}' \, = \, G_{x_{1}}$.  We recall ${\mathcal K}/G_{x_{0},1} \, \simeq \, \SL (2,\myFF ) \, \simeq \, {\mathcal K}'/G_{x_{1},1}$.

\smallskip

\begin{prop} \hfil

\smallskip

\begin{itemize}
\item[(i)] If $\kappa$ is an irreducible cuspidal representation of $\SL (2,\myFF )$, let $\kappa_{x_{0}}$ and $\kappa_{x_{1}}$ denote its inflation to ${\mathcal K} = \mG_{x_{0}}$ and ${\mathcal K}' = \mG_{x_{1}}$ respectively.  \ Then, the representations $\mycInd{\mG}{\mathcal K} (\kappa_{x_{0}} )$ and $\mycInd{\mG}{{\mathcal K}'} (\kappa_{x_{1}} )$ are irreducible supercuspidal representations of $\mG$.  Furthermore, the supercuspidal representations induced from the group ${\mathcal K}$ are inequivalent to those induced from the group ${\mathcal K}'$.

\smallskip

\item[(ii)] Any irreducible supercuspidal representation $(\pi , V_{\pi})$ of depth zero is equivalent to a $\mycInd{\mG}{\mathcal K} (\kappa_{x_{0}} )$ or a $\mycInd{\mG}{{\mathcal K}'} (\kappa_{x_{1}} )$

\smallskip

\item[(iii)] Normalize Haar measure on $\SL (2)$ so $\meas ( {\mathcal K} ) \, = \, 1$. If $\kappa$ is a cuspidal representation of $\SL (2,\myFF )$ inflated to ${\mathcal K}$, and $\pi \, = \, \mycInd{\mG}{{\mathcal K}} ( \kappa )$, then 
$$
{\text{\rm{the formal degree}}} \ \ d_{\pi} \ \ = \ \ {\text{\rm{degree}}}(\kappa ) \ . 
$$ 
\noindent{Similarly} for ${\mathcal K}'$.

\end{itemize}
\end{prop} 

\smallskip
We briefly review the cuspidal representations of $\SL (2,\myFF )$:

\smallskip

\begin{itemize}
\item[$\bullet$]  Let $T \subset \SL (2,\myFF )$ be the (elliptic) torus of order $(q+1)$.  Take $\beta$ to be a primitive root of $T$, set $\zeta \, = \, e^{\frac{2 \pi \sqrt{-1}}{(q+1)}}$, and for integral $0 \, \le \, i \, \le \, q$, let $\theta_{i} \in \widehat{T}$ be the character $\theta_{i}(\beta ) \, = \, \zeta^{i}$.  For $0 \, < \, i$, the character $\theta_{i}$ is conjugate under the normalizer $N_{\mG}(T)$ to $\theta_{(q+1)-i}$.

\smallskip

\item[$\bullet$] For $i = 1, 2, \dots , q$, there is a cuspidal representation 
$\sigma_{i}$ of $\SL (2, \myFF )$ of dimension $(q-1)$.  Let $\chi_{i}$ be the character of $\sigma_i$.   Then,  $\chi_{i} = \chi_{(q+1)-i}$.  For $i \neq {\frac{(q+1)}{2}}$, the character is irreducible, and for $i = {\frac{(q+1)}{2}}$, the character is the sum of two irreducible characters $\eta_{1}, \, \eta_{2}$ of degree ${\frac{q-1}{2}}$.  \ Let $\epsilon \in \myFF^{\times}$ be a non-square.  In all cases, the support of $\chi_{i}$ is on the elements: 

\smallskip

\begin{equation}
\myId \ , \quad -\myId \ , \quad \beta^k \ , \quad {\text{\rm{$u_1 = \left[ \begin{matrix} 1 &1 \\ 0 &1 \end{matrix} \right]$ \ , \quad $u_{\epsilon} = \left[ \begin{matrix} 1 &\epsilon \\ 0 &1 \end{matrix} \right]$ \ , \quad $-u_{1}$ \ , \quad $-u_{\epsilon}$ }}}  
\end{equation}

\noindent with values:

\medskip

\end{itemize}

\begin{center}
\begin{tabular}{ | c | c | c | c | c | c | c | c | }
    \hline 
${\begin{matrix} \ \\ \ \end{matrix}}$ &\myId &-\myId &$\beta^k$ &$u_1$ &$u_{\epsilon}$ &$-u_1$ &$-u_{\epsilon}$ \\
    \hline
${\begin{matrix} \ \\ \ \end{matrix}} \chi_{i}$ &$(q-1)$ &$(-1)^i (q-1)$ 
&$-( \, \zeta^{ik} \ + \ \zeta^{-ik} \, )$ &$-1$ &$-1$ &$(-1)^{(i+1)}$ &$(-1)^{(i+1)}$ \\
    \hline
\end{tabular}
\end{center}

\medskip

\noindent{Set} $\pi_{i} := \mycInd{\mG}{\mathcal K}(\sigma_{i})$, and let $\Theta_{i}$ be the character of $\pi_{i}$.   Note that $\pi_{i}(-\myId )$ is $(-1)^i$ times the identify operator.   If we use ${\mathcal K}'$ instead of ${\mathcal K}$, we can define analogous representations $\pi_{i}'$ and characters $\Theta_{i}'$.  Set 

\begin{equation}
\aligned
e_{\mathcal K} \ :&= \ ( \, q \, - \, 1 \, ) \ {\underset {1 \le i \le {\frac{(q-1)}{2}}} \sum }  \Theta_{i} \ \ + \ \ {\frac{(q-1)}{2}} \ \Theta_{({\frac{(q+1)}{2}})}\ = \  {\frac{(q-1)}{2}} \  {\underset {1 \le i \le q} \sum } \  \Theta_{i} \\
e_{{\mathcal K}'} \ :&= \ ( \, q \, - \, 1 \, ) \ {\underset {1 \le i \le {\frac{(q-1)}{2}}} \sum }  \Theta_{i}' \ \ + \ \ {\frac{(q-1)}{2}} \ \Theta_{({\frac{(q+1)}{2}})}' \ = \  {\frac{(q-1)}{2}} \  {\underset {1 \le i \le q} \sum } \  \Theta_{i}' \ , 
\endaligned
\end{equation}
\noindent{so} $e_{0} \, = \, e^{\text{\rm{PS}}}_{0} \, + \, e_{\mathcal K} \, + \, e_{{\mathcal K}'}$. 

\bigskip

\subsection{$e_{0}$ on split regular elements}  \ \  Suppose $y \in \mG$ is a regular split compact element with eigenvalues $\alpha, \, \alpha^{-1}$.  From Tables 2 and 3 (pages 1235-1236) of Sally-Shalika [{\reSSa}]:

\begin{equation}
\Theta_{i} (y) \ = \ \begin{cases}
\ 0 &{\text{\rm{when $\alpha \not\equiv \alpha^{-1} {\text{\rm{ mod }}} \wp$}}} \\
\ \\
\ {\dfrac{1}{| \, \alpha \, - \, \alpha^{-1} \, |}} \ - \ 1 &{\text{\rm{when $\alpha \equiv 1 {\text{\rm{ mod }}} \wp$}}} 
\end{cases}
\end{equation}

\noindent{Similarly} for the character $\Theta_{i}'$.  We deduce 
 
\begin{equation}
e_{\mathcal K} \, (y) \ = {\dfrac{(q-1)}{2}} \  \ \begin{cases}
\ 0 &{\text{\rm{when $\alpha \not\equiv \alpha^{-1} {\text{\rm{ mod }}} \wp$}}} \\
\ \\
\ (-1) \, \Big( \, {\dfrac{1}{| \, \alpha \, - \, \alpha^{-1} \, |}} \ - \ 1 \, \Big) &{\text{\rm{when $\alpha \equiv -1 {\text{\rm{ mod }}} \wp$}}} \\
\ \\
\ q \, \Big( \, {\dfrac{1}{| \, \alpha \, - \, \alpha^{-1} \, |}} \ - \ 1 \, \Big) &{\text{\rm{when $\alpha \equiv 1 {\text{\rm{ mod }}} \wp$}}} \ ,
\end{cases}
\end{equation}
\noindent{and} whence
\begin{equation}
e_{0}(y) \ = \ (q-1) \, (q+1) \ \begin{cases}
\ 0 &{\text{\rm{unless $\alpha \equiv 1 {\text{\rm{ mod }}} \wp$}}} \\
\\
\ \Big( \, {\dfrac{2}{| \, \alpha \, - \, \alpha^{-1} \, |}} \ - \ 1 \, \Big) &{\text{\rm{when $\alpha \equiv 1 {\text{\rm{ mod }}} \wp$}}} \ .
\end{cases} 
\end{equation}

\bigskip

\bigskip


\subsection{$e_{0}$ on ramified elliptic elements}  \ \  Suppose $y \in \mG$ is a ramified elliptic element, i.e., its eigenvalues $\alpha, \, \alpha^{-1}$ belong to a ramified quadratic extension $E$.  We note either $y$ or $-y$ is topologically unipotent.  We assume $y$ topologically unipotent, i.e., $\alpha \, \equiv \, 1 \ {\text{\rm{mod}}} \ \wp_{E}$.  From Table 2 (page 1235) of Sally-Shalika [{\reSSa}]:

\begin{equation}
\Theta_{i} ( y ) \ = \ -1 \ \ , \quad {\text{\rm{therefore}}} \quad \Theta_{i} ( -y ) \ = \ -1 \, (-1)^i \ ,
\end{equation}
\noindent{and so} 
\begin{equation}
\aligned
e^{\text{\rm{cusp}}}_{0}(y) \ &= \ -(q-1)q \qquad {\text{\rm{and}}} \qquad e^{\text{\rm{cusp}}}_{0}(-y) \ &= \ (q-1) \ .
\endaligned
\end{equation} 

\noindent{We deduce}

\begin{equation}
{\dfrac{1}{(q-1)(q+1)}} \, e_{0} (y) \ = \ 
\begin{cases}
\ 0 &{\text{\rm{ when  $y$ is not topologically unipotent}}}  \\
\ \\
\ -1 &{\text{\rm{ when  $y$ is topologically unipotent}}} \ .
\end{cases}
\end{equation}

\bigskip

\subsection{$e_{0}$ on unramified elliptic elements}  \ \  Suppose $y \in \mG$ is a unramified elliptic element, i.e., its eigenvalues $\alpha, \, \alpha^{-1}$ belong to a unramified quadratic extension $E$.  We consider the two cases depending on whether the two eigenvalues are congruent modulo $\wp_{E}$. 

\medskip

{\fontfamily{phv}\selectfont{Case $\alpha \ \not\equiv \ \alpha^{-1} \ {\text{\rm{mod}}} \ \wp_{E}${\,}:}}  \ \ We note there are two $\mG$-conjugacy classes of elements which have eigenvalues $\alpha, \, \alpha^{-1}$.  One class has non-empty intersection with ${\mathcal K}$, and empty intersection with ${\mathcal K}'$, and vice versa for the other class.  We assume $y \in {\mathcal K}$, and find,   using Harish-Chandra's formula [{\reHCa}] for the character of the compactly induced representation:
$$
\aligned
e_{\mathcal K} \, (y) \ &= \ {\dfrac{(q-1)}{2}} \ \mycInd{\mG}{\mathcal K} \ \big( {\underset {1 \le i \le q} \sum } \Theta_{\sigma_{i}} \big) \ (y) \ = \ (q-1) \ \ , \ \ {\text{\rm{and}}} \ \ e_{{\mathcal K}'} \, (y) \ &= \ 0 \ .
\endaligned
$$
\noindent{If} $y \in {\mathcal K}'$, we have the obvious analogous situation.   Therefore,  
$$
e_{0}(y) \ = \ ( \, e^{\text{\rm{PS}}}_{0} \, + \, e_{\mathcal K} \, + \, e_{{\mathcal K}'} \, ) \, (y) \ = \ (q-1) \, - \, (q-1) \ =  \ 0 \ \quad {\text{\rm{when \  $\alpha \ \not\equiv \ \alpha^{-1} \ {\text{\rm{mod}}} \ \wp_{E}$}}}\ .
$$

\medskip

{\fontfamily{phv}\selectfont{Case $\alpha \ \equiv \ \alpha^{-1} \ {\text{\rm{mod}}} \ \wp_{E}${\,}:}}  \ \  Here $\alpha \, \equiv \, \pm 1 \ {\text{\rm{mod}}} \ \wp_{E}$.  We assume $\alpha \, \equiv \, 1$, so $y$ is topologically unipotent.  Let $\sgn_{E}$ denote the class-field character of $\mF^{\times}$ of the unramified quadratic extension $E$.  By Table 2 of Sally-Shalika [{\reSSa}], the values of $\Theta_{i}(y)$ and $\Theta_{i}'(y)$ in some order are:
$$
-1 \ \pm \ \sgn_{E} (\gamma ) \ {\dfrac{1}{| \, \alpha \, - \, \alpha^{-1} \, |}} \ .
$$ 
\noindent{Here} $\gamma \, = \, \pm \alpha \, - \, \alpha^{-1}$ is the (imaginary) part of $\alpha$.  Note $\sgn_{E}(\pm 1) \, = \, 1$. So, 

$$
( \, \Theta_{i} \, + \, \Theta_{i}' \,) \, (y) \ = \ 2 \quad {\text{\rm{and}}} \quad ( \, \Theta_{i} \, + \, \Theta_{i}' \,) \, (-y) \ = \ (-1)^{i}2 \ .
$$ 

\noindent{It} follows:

$$
\aligned
( \, e_{\mathcal K} \, + \, e_{{\mathcal K}'} \, ) \, (y) \ &= \ {\dfrac{(q-1)}{2}} \, q \, (-2) \ = \ - \, (q-1) \, q \ \ , \ \ {\text{\rm{and}}} \\
( \, e_{\mathcal K} \, + \, e_{{\mathcal K}'} \, ) \, (-y) \ &= \ {\dfrac{(q-1)}{2}} \, (-1) \, (-2) \ = \ (q-1) \ .
\endaligned
$$

\noindent{Thus},

$$
\aligned
e_{0} \, (y) \ &= \ - (q-1)(q+1) \ \ , \ \ {\text{\rm{and}}} \quad e_{0} \, (-y) \ &= \ 0 \quad {\text{\rm{when $\alpha \, \equiv \, 1 \ {\text{\rm{mod}}} \ \wp_{E}$}}} \, .
\endaligned 
$$

\bigskip

\eject

\subsection{Summary Table of $e_{0}$} \hfil

\vskip 0.10in

\begin{center}
\begin{tabular}{ | c | l | }
    \hline 
\multicolumn{2}{| c |}{$\begin{matrix} {\hskip 6.0in} \\ {\text{\rm{ {\fontfamily{phv}\selectfont{Table 1}} \quad depth zero }}} \\ \ \end{matrix}$} \\
    \hline
\multicolumn{2}{| c |}{$y$ has eigenvalues $\alpha$, $\alpha^{-1}$ \ : \qquad value of \ ${\dfrac{1}{(q-1)(q+1)}} \ e_{0}(y) \begin{matrix} \ \\ \ \\ \ \end{matrix}$}   \\
    \hline
    \hline
\ &\ \\
{$y$ split} &$
\begin{cases}
\ 0 &{\text{\rm{when $\alpha \not\equiv 1 {\text{\rm{ mod }}} \wp$}}} \\
\\
\ \Big( \, {\dfrac{2}{| \, \alpha \, - \, \alpha^{-1} \, |}} \ - \ 1 \, \Big) &{\text{\rm{when $\alpha \equiv 1 {\text{\rm{ mod }}} \wp$}}} \ .
\end{cases}$
 \\
\ &\ \\
    \hline
\ &\ \\
{$\begin{matrix}  {\text{\rm{$y$ elliptic}}} \\ {\text{\rm{or unramified}}} \end{matrix}$} &$\begin{cases}
\ 0 &{\text{\rm{ when  $y$ is not topologically unipotent}}}  \\
\ \\
\ -1 &{\text{\rm{ when  $y$ is topologically unipotent}}} \ .
\end{cases}$
 \\
\ &\ \\
    \hline
\end{tabular} {\hfil} 
\end{center}


\vskip 1.00in 


\section{Integral depth supercuspidal representations}

\medskip

\subsection{The partition of supercuspidal representations by the groups ${\mathcal K}$ and ${\mathcal K}'$} \quad We abbreviate the filtration subgroups ${\mathcal K} \, = \, G_{{x_{0}},r}$ as ${\mathcal K}_{r}$ ($r \in {\mathbb N}$), and similarly we abbreviate the filtration subgroups of ${\mathcal K}' \, = \, G_{x_{1}}$.

\medskip

Suppose $d$ is a positive integer.   We recall if $(\pi , V_{\pi})$ is an irreducible supercuspidal representation $(\pi , V_{\pi})$ of depth $d$, then either 
$$
V^{{\mathcal K}_{d+1}}_{\pi} \ \neq \ \{ \, 0 \, \} \quad , \ \ {\text{\rm{(exclusive) or}}} \qquad V^{{\mathcal K}_{d+1}'}_{\pi} \ \neq \ \{ \, 0 \, \}
$$
In what follows, we assume $V^{{\mathcal K}_{d+1}}_{\pi} \ \neq \ \{ \, 0 \, \}$ and remark the obvious transposition of results holds when  $V^{{\mathcal K}_{d+1}'}_{\pi} \ \neq \ \{ \, 0 \, \}$.

\medskip

The subgroup ${{\mathcal K}_{d+1}}$ is normal in ${\mathcal K}$, and therefore 
$V^{{\mathcal K}_{d+1}}_{\pi}$ is ${\mathcal K}$-invariant, and so ${\mathcal K}_{d}$-invariant.  We note $({\mathcal K}_{d}/{\mathcal K}_{d+1})^{\widehat{\ }} \ = \ \fkg_{x_{0},-d}/\fkg_{x_{0},(-d+1)} \ \simeq\ \myFFsl2$, and we recall:

\smallskip

\begin{itemize} 
\item[(i)] If $(\pi , V_{\pi})$ is an irreducible supercuspidal representation of depth $d$ with $V^{{\mathcal K}_{d+1}}_{\pi} \ \neq \ \{ \, 0 \, \}$, then the characters of ${\mathcal K}_{d}/{\mathcal K}_{d+1}$ which appear $V^{{\mathcal K}_{d+1}}_{\pi}$ have the form $\phi_{\Xi}$ where the coset $\Xi \in \fkg_{x_{0},-d}/\fkg_{x_{0},(-d+1)} \ \simeq\ \myFFsl2$ is an elliptic element in  $\myFFsl2$, i.e., has irreducible characteristic polynomial.  By Clifford theory, the set of $\phi_{\Xi}$'s which appear in $V^{{\mathcal K}_{d+1}}_{\pi}$ form a single ${\mathcal K}$-orbit in $\fkg_{x_{0},-d}/\fkg_{x_{0},(-d+1)}$.

\smallskip

\item[(ii)] The Adjoint action of ${\mathcal K}$ on $\myFFsl2$ has $\dfrac{(q-1)}{2}$ orbits of elliptic elements, and each orbit contains $(q-1)q$ elliptic elements.

\smallskip

\item[(iii)] If $\phi_{\Xi}$ is a character of ${\mathcal K}_{d}/{\mathcal K}_{d+1}$ attached to an elliptic coset of $\Xi \in \fkg_{x_{0},-d}/\fkg_{x_{0},(-d+1)} \ \simeq\ \myFFsl2$, then 
$$
\mycInd{\mG}{{\mathcal K}_d} \ ( \, \phi_{\Xi} \, )
$$

\noindent{is} a finite length (completely reducible) supercupsidal representation.   If $(\pi , V_{\pi})$  is as in part (i), i.e.,  $V^{{\mathcal I}_{d^{+}}}_{\pi}$ contains the (non-degenerate) character $\phi_{\Xi}$, then by Frobenius reciprocity:

$$
\Hom_{\mG} ( \, V_{\pi} \, , \, \mycInd{\mG}{{\mathcal K}_{d}} \ (  \phi_{\Xi}  ) \, ) \ \neq \ \{ \, 0 \, \} \ .
$$

\noindent{Furthermore}: 

\smallskip

\begin{itemize} 
\item[$\bullet$] Up to isomorphism, $\mycInd{\mG}{{\mathcal K}_{d}} \ ( \phi_{\Xi}  )$ contains $(q+1)q^{(d-1)}$ {distinct} classes of irreducible supercuspidal representations $(\sigma , V_{\sigma})$.  In particular, the number of distinct irreducible supercuspidal representations of depth $d$ which are induced from ${\mathcal K}$ is:

\begin{equation}\label{integral-depth-numbers}
{\dfrac{(q-1)}{2}}  \, (q+1) \, q^{(d-1)}
\end{equation}

\smallskip

\item[$\bullet$] The formal degree of any irreducible $(\sigma , V_{\sigma} )$ occurring in $\mycInd{\mG}{{\mathcal K}_{d}} \ ( \phi_{\Xi}  )$ is:

\begin{equation}\label{integral-depth-formal-degree}
d_{\sigma} \ = \ {\dfrac{1}{\meas ({\mathcal K} )}} \, (q-1) \, q \, q^{(d-1)} 
\end{equation}

\smallskip

\item[$\bullet$] The multiplicity in $\mycInd{\mG}{{\mathcal K}_{d}} \ ( \phi_{\Xi}  )$ of any $(\sigma , V_{\sigma} )$ occurring in it is $q^{(d-1)}$; in particular is independent of $\sigma$.

\end{itemize}

\smallskip

\item[(iv)] For an elliptic character $\phi_{\Xi}$ of ${\mathcal K}_{d}/{\mathcal K}_{d+1}$, let $S_{\Xi}$ denote the set of $(q+1) \, q^{(d-1)}$ classes of irreducible supercuspidal representations occurring in $\mycInd{\mG}{{\mathcal K}_{d}} \ ( \phi_{\Xi}  )$, and let $\Theta_{\Xi}$ denote the character of the representation $\mycInd{\mG}{{\mathcal I}_{d}} \ ( \phi_{\Xi}  )$.  We have:

$$
\Theta_{\Xi} \ = \ q^{(d-1)} \ {\underset {\sigma \in S_{\Xi}} \sum} \ \Theta_{\sigma} \ .
$$
\noindent{By} Harish-Chandra's character formula [{\reHCa}] for induction from an open compact subgroup, $\Theta_{\Xi}$ is supported on $\myAd (\myG) \, ({\mathcal K}_{d})$.  If the cosets $\Xi_1$ and $\Xi_2$ belong to the same ${\mathcal K}$-orbit, then $\mycInd{\mG}{{\mathcal I}_{d}} \ ( \phi_{\Xi_1}  )$ and $\mycInd{\mG}{{\mathcal I}_{d}} \ ( \phi_{\Xi_2}  )$ are equivalent representations, and therefore $\Theta_{\Xi_1} \ = \ \Theta_{\Xi_2}$.

\end{itemize}

\medskip

Set 

\begin{equation}\label{sum-unramified-characters}
\tau \ := \ {\underset {\Xi} \sum } \ \phi_{\Xi} \quad {\text{\rm{the sum of the elliptic characters of ${\mathcal K}_{d}$ (modulo ${\mathcal K}_{d+1}$),}}}
\end{equation}

\noindent{and} let $\Theta_{\tau}$ denote the character of the induced representation $\mycInd{\myG}{{\mathcal K}_{d}} (\tau )$. \  We have:

\begin{equation}\label{cusp-integral-sum-a}
\aligned
\Theta_{\tau} \ &= \ {\underset {\Xi} \sum } \ \Theta_{\Xi} \ = \ {\underset {\Xi} \sum } \ q^{(d-1)} \ {\underset {\sigma \in {S_{\Xi}}} \sum } \ \Theta_{\sigma}  \ = \ (q-1) \, q \ q^{(d-1)} \ {\underset {\rho(\sigma ) \, = \, d} {{\sum}_{\mathcal K}} } \Theta_{\sigma} \ , \\
\endaligned
\end{equation}

\noindent{and} so,

\begin{equation}\label{cusp-integral-sum-b}
\aligned
{\dfrac{\Theta_{\tau}}{\meas ({\mathcal K})}} \ &= \ {\dfrac{(q-1) \, q \ q^{(d-1)}}{\meas ( {\mathcal K} )}}  \ {\underset {\rho(\sigma ) \, = \, d} {{\sum}_{\mathcal K}} } \Theta_{\sigma} \ = \ {\underset {\rho(\sigma ) \, = \, d} {{\sum}_{\mathcal K}} } d_{\sigma} \, \Theta_{\sigma} \ .\\
\endaligned
\end{equation} 

\noindent{The} sum ${{\sum}_{\mathcal K}}$ is over the classes of irreducible supercuspidal representations of depth $d$ which can be induced from ${\mathcal K}$.  
\medskip

For the compact group ${\mathcal K}'$, with the obviously transposed construction, we have the similar identity: 

\begin{equation}\label{cusp-integral-sum-c}
\aligned
{\dfrac{\Theta_{{\tau}'}}{\meas ({\mathcal K}')}} \ &= \ {\underset {\rho(\sigma ) \, = \, d} {{\sum}_{{\mathcal K}'}} } d_{\sigma} \, \Theta_{\sigma} \ .\\
\endaligned
\end{equation} 

\noindent{Whence}, 

\begin{equation}\label{cusp-integral-sum-d}
\aligned
e^{\text{\rm{cusp}}}_{d} \ &= \ {\underset {\rho(\sigma ) \, = \, d} {{\sum}_{{\mathcal K}'}} } d_{\sigma} \, \Theta_{\sigma} \ + \ {\underset {\rho(\sigma ) \, = \, d} {{\sum}_{\mathcal K}} } d_{\sigma} \, \Theta_{\sigma} \ = \ {\dfrac{1}{\meas ({\mathcal K})}} \ ( \, \Theta_{\tau} \ + \ \Theta_{{\tau}'} \, )
\endaligned
\end{equation}

\bigskip

\subsection{The projector $e^{\text{\rm{cusp}}}_{d}$} \quad Let $Z$ denote the subgroup of scalar matrices in $\GL (2,\mF )$, and set

\begin{equation}
\aligned
K \ :&= \ \GL (2,{\mathcal R}_{\mF} ) \\
\fkk :&= \ \fkg \fkl   (2,{\mathcal R}_{\mF} ) \\
\mG_{u} \ :&= \ \{ \ g \in \GL (2,\mF ) \ | \ {\text{\rm{$\myval_{F}( \, \det (g) \, )$ is even}}} \ \} \ = \ Z \, K  \, \SL (2,\mF ) \ . \\
\endaligned
\end{equation}

\noindent{Let} $K_n$ be the usual congruence subgroup of $K$, and for $n \, \in \, {\mathbb Z}$, set $\fkk_{n} \ = \ \varpi^{n} \fkk$.  In particular, there are natural maps 
$$
K_{d}/K_{d+1} \ \simeq \ {\fkk_d}/{\fkk_{d+1}} \quad {\text{\rm{and}}} \quad \big( \, \fkk_{d}/\fkk_{d+1} \, \big)^{\widehat{\ }} \ = \ {\fkk_{-d}}/{\fkk_{-d+1}} \ \simeq \ \fkg \fkl   (2,\myFF ) \ .
$$
\noindent{The} inclusion $\fks \fkl   (2,\myFF ) \ \subset \ \fkg \fkl   (2,\myFF )$ allows us to canonically extend the previous subsection characters $\phi_{\Xi}$ of ${\mathcal K}_d$ (trivial ${\mathcal K}_{d+1}$),  to characters of $ZK_{d}$ (trivial $Z{\mathcal K}_{d+1}$).  We do this, and consequently get a canonical extension of the character $\tau$ to $ZK_{d}$.  We shall use the same symbol $\tau$ to denote the extension.  We recall:
\begin{equation}
{\text{\rm{Res}}}^{\GL (2, \mF )}_{\SL (2, \mF )} \ \big( \, \mycInd{\GL (2, \mF )}{ZK_{d}} \ (\tau ) \ \big) \ = \ {\text{\rm{c-Ind}}}^{\SL (2, \mF )}_{{\mathcal K}_{d}} \, (\tau ) \ \oplus \  {\text{\rm{c-Ind}}}^{\SL (2, \mF )}_{{{\mathcal K}'}_{d}} \, ({\tau}') \ ,
\end{equation}


\noindent{and} therefore,  

\begin{equation}\label{cusp-integral-sum-e}
\aligned
e^{\text{\rm{cusp}}}_{d} \ &= \ {\dfrac{1}{\meas ({\mathcal K})}} \ \cdot \ {\text{\rm{ restriction to $\SL (2, \mF )$ of the character of $\mycInd{\GL (2, \mF )}{ZK_{d}} \ 
(  \, \tau \, )$}}} \ . 
\endaligned
\end{equation}

\bigskip

\subsection{The projector $e^{\text{\rm{cusp}}}_{d}$ on split and ramified elliptic tori} \quad

\medskip

If $y \, \in \, \mG$ is a regular element which is split or ramified elliptic, and $\pi$ is an irreducible supercuspidal representation of depth $d$, the Sally-Shalika [{\reSSa}] character tables lists the character value $\Theta_{\pi} (y)$ as zero unless the eigenvalues $\alpha, \, \alpha^{-1}$ of $y$ have $| \, \alpha \, - \, \alpha^{-1} \, | \ < \ q^{-d}$.  Furthermore, when  $| \, \alpha \, - \, \alpha^{-1} \, | \ < \ q^{-d}$, the value $\Theta_{\pi} (y)$ depends only on $| \, \alpha \, - \, \alpha^{-1} \, |$ and $d$.  By \eqref{integral-depth-numbers} and \eqref{integral-depth-formal-degree}, it follows:

\begin{equation}
\aligned
e^{\text{\rm{cusp}}}_{d} \, (y) \ &= \ {\text{\rm{\# irreducible supercuspidals }}} \ \cdot \ {\text{\rm{ formal degree}}} \ \cdot  \ \Theta_{\pi} (y) \\ 
&= \ (q-1) (q+1) q^{(d-1)} \ {\dfrac{1}{\meas ({\mathcal K})}} \, (q-1) \, q \, q^{(d-1)} \ \Theta_{\pi}(y)   \ .
\endaligned
\end{equation}

{\fontfamily{phv}\selectfont{Case $y$ split:}} \ \ If $y$ is a (compact) split element with eigenvalues $\alpha, \, \alpha^{-1}$, then, by Table 2 (page 1235) of Sally-Shalika [{\reSSa}], the character value of a depth $d$ irreducible supercuspidal (unramified) representation $\pi$ is: 
$$
\Theta_{\pi} \, (y) \ = \ \begin{cases}
\ \ 0 \qquad &{\text{\rm{when \ \ $\alpha \ \notin \ 1+{\wp}^{d+1}_{F}$}}}    \\
\ \\
\ \ {\dfrac{1}{| \, \alpha \, - \, \alpha^{-1} \, |}} \ - \ q^{d}  &{\text{\rm{for $\alpha \ \in \ ( \, 1+{\wp}^{d+1}_{F} \, )$ \ \ ; }}}
\end{cases}
$$ 
\noindent{so}, we get 
$$
e^{\text{\rm{cusp}}}_{d} \, (y) \ = \ (q-1) \ (q+1) \ 
\begin{cases}
\ 0 {\hskip 0.50in} {\text{\rm{when $\alpha \ \notin 1+{\wp}^{d+1}_{F}$, i.e., \ $| \, 1 \, - \, \alpha \, | \ > \ {\frac{1}{q^{d+1}}}$}}} \\
\ \\
\ (q-1) \ q \ q^{2(d-1)} \ \Big( \, {\dfrac{1}{| \, \alpha \, - \, \alpha^{-1} \, |}} \ - \ q^{d} \, \Big) {\hskip 0.10in} {\text{\rm{for \ $| \, 1 \, - \, \alpha \, | \ \le \ {\frac{1}{q^{d+1}}}$ \  ;}}} \\
\end{cases}
$$ 

\medskip

{\fontfamily{phv}\selectfont{Case $y$ ramified elliptic:}} \ \ If $y$ is a ramified elliptic element with eigenvalues $\alpha, \, \alpha^{-1}$, let $E$ be the quadratic extension of $\mF$ containing $\alpha$.  Then, the character value of a depth $d$ cuspidal (unramified) representation $\pi$ is 

$$
\Theta_{\pi} \, (y) \ = \ \begin{cases}
\ \ 0 &{\text{\rm{when $\alpha \ \notin \ (1+{\wp}^{2d+1}_{E})$}}} \\
\ \\
\ \ - \ q^{d}  &{\text{\rm{for $\alpha \ \in \ ( \, 1+{\wp}^{2d+1}_{E} \, )$, i.e., \ $| \, 1 \, - \, \alpha \, | \ \le \ {\frac{1}{q^{d+{\frac12}}}}$, }}} \\
\end{cases}
$$
 
\noindent{so}, we get 
$$
e^{\text{\rm{cusp}}}_{d} \, (y) \ = \ (q-1) \ (q+1) \ \begin{cases}
\ \ 0 &{\text{\rm{when $\alpha \ \notin 1+{\wp}^{2d+1}_{E}$,}}} \\
\ \\
\ \ (q-1) \ q \ q^{2(d-1)} \ q^{d} \ (-1)  &{\text{\rm{for \ $| \, 1 \, - \, \alpha \, |_{E} \ \le \ {\frac{1}{q^{d+{\frac12}}}}$  \ .}}} \\
\end{cases}
$$ 

\bigskip

\subsection{The projector $e^{\text{\rm{cusp}}}_{d}$ on unramified elliptic tori} \quad Suppose $y \, \in \, \mG$ is a regular element whose eigenvalues $\alpha, \, \alpha^{-1}$ bekong to the unramified quadratic extension $E/F$.  We deduce from the formula \eqref{cusp-integral-sum-e} that 
$$
e^{\text{\rm{cusp}}}_{d} \, (y) \ = \ 0  \qquad {\text{\rm{unless $\alpha \, \in 1+{\wp}^{d}_{E}$ .}}}
$$

\bigskip

{\fontfamily{phv}\selectfont{Subcase \ $| \, \alpha - \alpha^{-1} \, | \, = \, q^{-d}$}}:    \ The element $y$ has a conjugate in the set $K_{d} \ \backslash \ K_{d+1}$.  We may and do assume $y$ is in $K_{d} \ \backslash \ K_{d+1}$.

\medskip

\begin{lemma} Let $\psi$ be a non-trivial character of $\myFF$.  Suppose $z \, \in \, \myFFsl2$ is a elliptic element.    Then, 
\begin{equation}\label{gauss-2}
{\underset {\begin{matrix} e \, \in \, \myFFsl2 \\ {\text{\rm{ $e$ elliptic }}} \end{matrix}} \sum} \psi \, (\, \mytrace ( \, z \, e \, ) \, ) \ = \ q \ . 
\end{equation}
\end{lemma}

\begin{proof}  The orbit of an elliptic element has a representive of the form 
$\left[ \begin{matrix} 0 &1 \\ u &0 \end{matrix} \right]$, with $u$ a non-square.  A parametrization of the elliptic elements is: 

$$
\myAd \big( \, \left[\begin{matrix} a &b \\ 0 &1 \end{matrix} \right] \big) \ \left[ \begin{matrix} 0 &1 \\ u &0 \end{matrix} \right]   \ \ , \ \ {\text{\rm{with $u$ a non-square, \ $a$ non-zero, \ and $b$ arbitrary.}}} 
$$

We take $z \, = \, \left[ \begin{matrix} 0 &1 \\ \epsilon &0 \end{matrix} \right]$, with $\epsilon$ a non-square.  Then, the sum \eqref{gauss-2} becomes

$$
\aligned
{\underset {\begin{matrix} {\text{\rm{$u$ non-square}}} \\  
 {a \, \neq \, 0 \ , \ b } 
\end{matrix}}
\sum } &\psi \, ( \ u \, a^{-1} \, + \, \epsilon \, a \, - \, \epsilon \, b^2 \, u \, a^{-1} \ ) \ = \ 
{\underset a \sum} \ \psi ( \, \epsilon a \, ) 
{\underset u \sum} \ \psi ( \, ua^{-1} \, ) \ {\underset b \sum} \ \psi ( \, - \epsilon b^2 u a^{-1} \, ) \\
&= \ {\underset a \sum} \ \psi ( \, \epsilon a \, ) 
{\underset {\text{\rm{$v$ non-square}}} \sum} \ \psi ( \, va \, ) \ {\underset c \sum} \ \psi ( \, - a c^2  \, ) \ . \\
\endaligned
$$

\noindent{Let} $\psi_{a}$ be the character $\psi_a(x) \, = \, \psi (ax)$, and let $\sgn$ denote the quadratic character on $\myFF$.  Then,  ${\underset c \sum} \ \psi ( \, - a c^2 a \, ) \, = \, G(\psi_{-a},\sgn ) \, = \, \sgn(-a) \, G(\psi , \sgn )$, a Gauss sum.  So, the sum \eqref{gauss-2} equals

$$
\aligned
{\underset {a \neq 0} \sum} \ \psi ( \, \epsilon a \, ) \ &\Big(  
{\underset {\text{\rm{$v$ non-square}}} \sum} \ \psi ( \, va \, ) \ \Big) \ \sgn(-a) \, G(\psi , \sgn )  \\
&= \ \ {\underset {a \neq 0} \sum} \ \ \psi ( \, \epsilon a \, ) \ \ {\frac12} \, \big( \, -1 \, - \, G(\psi_{a},\sgn ) \, \big)  \ \ \sgn(-a) \, G(\psi , \sgn ) \\
&= \ \ -{\frac12} \ \ {\underset {a \neq 0} \sum} \ \big( \ \psi( \epsilon a) \, \sgn(-a) \,  G(\psi , \sgn ) \ + \ \psi( \epsilon a) \, q \ \big) \\
&= \ \ -{\frac12} \ \ \big( \ G(\psi , \sgn ) \, \sgn (-1) \, \sgn( \epsilon ) \, {\underset {a \neq 0} \sum} \ \psi(a) \, \sgn (a)  \ \ + \ \ {\underset {a \neq 0} \sum} \ \psi( \epsilon a) \, q \ \big) \\
&= \ \ -{\frac12} \ \ \big( \ -q \ + \ (-q) \ \big) \ = \ q 
\endaligned
$$
\noindent{Here}, we have used elementary properties of Gauss sums [{\reIR}].
\end{proof}

\bigskip

\begin{cor} Let $\tau$ be the sum \eqref{sum-unramified-characters} of the elliptic characters of ${\mathcal K}_{d}$ (modulo ${\mathcal K}_{d+1}$).  If $y \in {\mathcal K}_{d} \ \backslash \ {\mathcal K}_{d+1}$ is an elliptic element whose eigenvalues $\alpha, \ \alpha^{-1}$ satisfy 
$| \, \alpha - \alpha^{-1} \, | \, = \, q^{-d}$, then 
$$
\tau \, ( \, y \, ) \ = \ q \ .
$$
\end{cor}

\vskip 0.20in

As already mentioned above \eqref{cusp-integral-sum-e}, $\tau$ has a canonical extension to $ZK_{d}$.  Let ${\overset {\text{\rm{\bf{.}}}} {\tau} }$  denote the extension of the function $\tau$ to $\GL (2,\mF )$  which is zero outside $ZK_d$.  Suppose $y \in {\mathcal K}_{d} \backslash {\mathcal K}_{d+1}$ is unramified elliptic.   Harish-Chandra's formula [{\reHCa}] for the right side of   \eqref{cusp-integral-sum-e}

\begin{equation}\label{character-value-cusp-integral}
\aligned
e^{\text{\rm{cusp}}}_{d}&(y) \ = \ {\underset {gZK_d \, \in \, \GL (2,\mF ) / ZK_d } \sum } \ {\overset {\text{\rm{\bf{.}}}} {\tau} } \, ( \, g \, y \, g^{-1} \, ) \  \qquad {\text{\rm{for $y \in {\mathcal K}_{d} \backslash {\mathcal K}_{d+1}$ elliptic}}} \\
&\ \\ 
&= \ {\underset {gZK_d \, \in \, ZK / ZK_d } \sum } \ {\overset {\text{\rm{\bf{.}}}} {\tau} } \, ( \, g \, y \, g^{-1} \, ) \ = \ {\underset {gZK_d \, \in \, ZK / ZK_d } \sum } \ {\overset {\text{\rm{\bf{.}}}} {\tau} } \, ( \, y \, ) \ = \ [ZK : ZK_d] \ \tau ( \, y \, ) \\
&\ \\
&= \ (q+1)(q-1)q \ (q^3)^{(d-1)} \tau (y) \ = \ (q+1) \, (q-1) \, q^2 \ q^{3(d-1)}
\endaligned
\end{equation}

\bigskip

{\fontfamily{phv}\selectfont{Subcase \ $| \, \alpha - \alpha^{-1} \, | \, < \, q^{-d}$}}:    Suppose $\pi$ is an irreducible supercupsidal representation compact induced from the representation $\kappa$ of ${\mathcal K}$.  Let $\kappa '$ be the $\GL(2,\mF )$ conjugation of $\kappa$ to a representation of ${\mathcal K}'$, and $\pi ' \, = \, \mycInd{\mG}{{\mathcal K}'} (\kappa ' )$.   By Table 2 (page 1235) of Sally-Shalika [{\reSSa}], 

$$
\Theta_{\pi} (y) \ + \  \Theta_{\pi '} (y) \ = \ -2 \, q^{d} \ ;
$$

\noindent{and} so 
$$
\aligned
e^{\text{\rm{cusp}}}_{d} &(y) \ = \ {\dfrac{(q-1)}{2}} \, (q+1) \, q^{(d-1)} \ \cdot \ (q-1)\, q \,  q^{(d-1)} \ \cdot \ (-2) \, q^{d}  \\
&= \ (q-1)(q+1) \ (q-1) q^{3d-1} \ (-1) \qquad {\text{\rm{when $| \, \alpha - \alpha^{-1} \, | \, < \, q^{-d}$}}} . \\
\endaligned
$$

\bigskip 
\bigskip


\begin{center}
\begin{tabular}{ | c | l | }
    \hline 
\multicolumn{2}{| c |}{ {\fontfamily{phv}\selectfont{Table 2}}$\begin{matrix} \ \\ \ \end{matrix}$ \quad $d \ \ge \ 0$ integral} \\
   \hline 
\multicolumn{2}{| c |}{$y$ has eigenvalues $\alpha$, $\alpha^{-1}$ \ : \quad value of \ ${\dfrac{1}{(q-1)(q+1)}} \ e^{\text{\rm{cusp}}}_{d}(y) \begin{matrix} \ \\ \ \\ \ \end{matrix}$}   \\
    \hline
    \hline
\ &\ \\
{$y$ split} &$\begin{cases}
\ 0 {\hskip 0.50in} {\text{\rm{ when $\alpha \ \notin 1+{\wp}^{d+1}_{F}$, i.e., \  $| \, 1 \, - \, \alpha \, | \ > \ {\frac{1}{q^{d+1}}}$}}} \\
\ \\ 
\ (q-1) \ q \ q^{2(d-1)} \ \Big( \, {\dfrac{1}{| \, \alpha \, - \, \alpha^{-1} \, |}} \ - \ q^{d} \, \Big) {\hskip 0.30in} {\text{\rm{for \ $| \, 1 \, - \, \alpha \, | \ \le \ {\frac{1}{q^{d+1}}}$ }}} \\
\end{cases}$
 \\
\ &\ \\
    \hline
\ &\ \\
{${\begin{matrix} 
{\text{\rm{$y$ ramified}}} \\
{\text{\rm{elliptic}}} \end{matrix}}$}
&$\begin{cases}
\ \ 0 {\hskip 0.50in} {\text{\rm{ when $\alpha \notin 1+{\wp}^{2d+1}_{E}$, i.e., \  $| \, 1 \, - \, \alpha \, | \ > \ {\frac{1}{q^{d+{\frac12}}}}$}}} \\
\ \\
\ \ (q-1) \ q \ q^{2(d-1)} \ q^{d} \ (-1)  {\hskip 0.30in} {\text{\rm{for \ $| \, 1 \, - \, \alpha \, |_{E} \ \le \ {\frac{1}{q^{d+{\frac12}}}}$  }}} \\
\end{cases}$
 \\
\ &\ \\
    \hline
\ &\ \\
{${\begin{matrix} 
{\text{\rm{$y$ unramified}}} \\
{\text{\rm{elliptic}}} \end{matrix}}$} 
&$\begin{cases}
\ 0 {\hskip 0.70in} {\text{\rm{ when $\alpha \ \notin 1+{\wp}^{d+1}_{E}$, i.e., \  $| \, 1 \, - \, \alpha \, | \ > \ {\frac{1}{q^{d}}}$}}} \\
\ \\ 
\ q^{(3d-1)}  {\hskip 0.40in} {\text{\rm{when $| \, 1 \, - \, \alpha \, | \ = \ {\frac{1}{q^{d}}}$}}} \\
\ \\ 
\ (q-1) \ q^{(3d-1)} \ (-1)  {\hskip 0.30in} {\text{\rm{for \ $| \, 1 \, - \, \alpha \, | \ < \ {\frac{1}{q^{d}}}$ }}} \\
\end{cases}$ \\
\ &\ \\
    \hline
\end{tabular} {\hfil} 
\end{center}

\vskip 1.0in



\section{Half-integral depth supercuspidal representations}

We abbreviate the Iwahori subgroup $G_{x_{01}}$ and its filtration subgroups $G_{{x_{01}},r}$ as ${\mathcal I}$ and ${\mathcal I}_{r}$ respectively.

\medskip

Suppose $d \, \in \, \frac12 \, + \, {\mathbb N}$ is a positive half-integer.  A Bernstein component $\Omega$ of depth $d$ is the equivalence class of an irreducible supercuspidal representation $\pi$.  Set
$$
d^{+} \ := \ d \ + \ \frac12 \ .
$$

\medskip

We recall: 

\medskip

\begin{itemize}

\item[(i)] The group  ${\mathcal I}_{d}/{\mathcal I}_{d^{+}}$ has $(q-1)^2$ non-degenerate characters.  Under the adjoint action of ${\mathcal I}$, these non-degenerate characters are partitioned into $2(q-1)$ orbits with ${\frac{(q-1)}{2}}$ characters in an orbit.

\smallskip

\item[(ii)] If $(\pi , V_{\pi})$ is an irreducible supercuspidal representation, and its depth $\rho (\pi )$ equals $d$, then the  subspace of ${\mathcal I}_{d^{+}}$-fixed vectors, $V^{{\mathcal I}_{d^{+}}}_{\pi}$, is non-zero, and is,  since ${\mathcal I}$ normalizes the subgroup ${\mathcal I}_{d^{+}}$, ${\mathcal I}$-invariant.  The characters $\phi_{X}$ of ${\mathcal I}_{d}$ (modulo ${\mathcal I}_{d^{+}}$) which appear in $V^{{\mathcal I}_{d^{+}}}_{\pi}$ are non-degenerate.  By Clifford theory, the set 
$$
\{ \ \phi_{\Xi} \ \ | \ \ \phi_{\Xi} \ \ {\text{\rm{appears in}}} \ \ V^{{\mathcal I}_{d^{+}}}_{\pi} \ \}
$$
is a single ${\mathcal I}$-orbit. 

\medskip

\noindent{The} formal degree of $\pi$ is: 

$$
d_{\pi} \ = \ {\frac{(q+1)}{\meas ({\mathcal K} )}} \ {\frac{(q-1)}{2}} \ q^{d - \frac12 } \ .
$$ 

\medskip

\item[(iii)] For any non-degenerate character $\phi_{\Xi}$ of ${\mathcal I}_{d}/{\mathcal I}_{d^{+}}$, the compactly supported induced representation
$$
\mycInd{\mG}{{\mathcal I}_{d}} \ ( \phi_{\Xi}  )
$$
is a finite length (completely reducible) supercuspidal representation. 
\medskip
\noindent{If} $(\pi , V_{\pi})$  is an irreducible supercuspidal representation as in part (i), i.e.,  $V^{{\mathcal I}_{d^{+}}}_{\pi}$ contains the (non-degenerate) character $\phi_{\Xi}$, then by Frobenius reciprocity:

$$
\Hom_{\mG} ( \, V_{\pi} \, , \, \mycInd{\mG}{{\mathcal I}_{d}} \ (  \phi_{\Xi}  ) \, ) \ \neq \ \{ \, 0 \, \} \ .
$$

\noindent{Furthermore}: 

\smallskip

\begin{itemize} 
\item[$\bullet$] Up to isomorphism, $\mycInd{\mG}{{\mathcal I}_{d}} \ ( \phi_{\Xi}  )$ contains $2 \, q^{(d-\frac12 )}$ {distinct} classes of irreducible supercuspidal representations $(\sigma , V_{\sigma})$. 

\smallskip

\item[$\bullet$] The multiplicity in $\mycInd{\mG}{{\mathcal I}_{d}} \ ( \phi_{\Xi}  )$ of any $(\sigma , V_{\sigma} )$ occurring in it is $q^{(d- \frac12 )}$; in particular, the multiplicity is independent of $\sigma$

\end{itemize} 

\medskip

\item[(iv)] For a non-degenerate character $\phi_{\Xi}$, let $S_{\Xi}$ denote the set of these $2 \, q^{(d-\frac12 )}$ classes of irreducible supercuspidal representations, and let $\Theta_{\Xi}$ denote the character of the representation $\mycInd{\mG}{{\mathcal I}_{d}} \ ( \phi_{\Xi}  )$.  We have:

$$
\Theta_{\Xi} \ = \ q^{(d-\frac12 )} \ {\underset {\sigma \in S_{\Xi}} \sum} \ \Theta_{\sigma} \ .
$$
\noindent{By} Harish-Chandra's character formula [{\reHCa}] for induction from an open compact subgroup, $\Theta_{\Xi}$ is supported on $\myAd (\myG) \, ({\mathcal I}_{d})$.  If the cosets $\Xi$ and $\Xi'$ belong to the same ${\mathcal I}$-orbit, then $\mycInd{\mG}{{\mathcal I}_{d}} \ ( \phi_{\Xi}  )$ and $\mycInd{\mG}{{\mathcal I}_{d}} \ ( \phi_{\Xi'}  )$ are equivalent representations, and so 
$\Theta_{\Xi} \ = \ \Theta_{\Xi'}$.  Set 

\begin{equation}
\tau \ := \ {\underset {\Xi} \sum } \ \phi_{\Xi} \quad {\text{\rm{the sum of the non-degenerate characters of ${\mathcal I}_{d}$ (modulo ${\mathcal I}_{d^{+}}$),}}}
\end{equation}

\noindent{and} let $\Theta_{\tau}$ denote the character of the (compactly supported) induced representation $\mycInd{\myG}{{\mathcal I}_{d}} (\tau )$.  We have: 

\begin{equation}\label{half-integral-sum-a}
\aligned
\Theta_{\tau} \ &= \ {\underset {\Xi} \sum } \ \Theta_{\Xi} \ = \ {\underset {\Xi} \sum } \ q^{(d-\frac12 )} \ {\underset {\sigma \in {S_{\Xi}}} \sum } \Theta_{\sigma} \\
&= \ {\frac{(q-1)}{2}} \, q^{(d-\frac12 )} \ {\underset {\rho(\sigma ) \, = \, d} \sum } \Theta_{\sigma} \ .
\endaligned
\end{equation} 

\noindent{Whence}, 

$$
\aligned
\Theta_{\tau} \ 
&= \ {\dfrac{\meas ({\mathcal K})}{(q+1)}}  \ \Big(  \, {\dfrac{(q+1)}{\meas ({\mathcal K})}} \ {\frac{(q-1)}{2}}  \, q^{(d - \frac12 )} \, \Big) \ {\underset {\rho(\sigma ) \, = \, d} \sum } \Theta_{\sigma} \ , 
\endaligned
$$

\noindent{i.e.,}

$$
\aligned
e_{d} \ &= \ {\dfrac{(q+1)}{\meas ({\mathcal K})}} \ \Theta_{\tau} \ . \\
\endaligned
$$

\end{itemize}

\medskip

\noindent{As} already mentioned above for $\Theta_{\Xi}$, by Harish-Chandra's formula [{\reHCa}] for supercuspidal representations obtained via compact induction, we have:
$$
{\text{\rm{support$(e_{d}) \ \subset \ \myAd (\mG ) \, ({\mathcal I}_{d})$}}} \ .
$$

\noindent{This} support condition allows us to computer $e_{d}$ rather efficiently.  Note for $d>0$, and residual characteristic $p$ odd, the set 
${\mathcal I}_d$ is contained in ${\mathcal U}^{\text{\rm{top}}}$.  Whence the support of $e_{d}$ is within the set of topologically unipotent elements.

\medskip

Suppose $y \in \mG$ is regular semisimple element.  Let $\alpha , \, \alpha^{-1}$ be the roots of the characteristic polynomial of $y$.  By the support condition:

$$
e_{d} ( \, y \, ) \ = \ 0 \qquad {\text{\rm{when \ \ \ $| \, \alpha \, - \, \alpha^{-1} \, | \ > \ q^{-d}$ \ . }}} 
$$

\medskip

\noindent{When} $| \, \alpha \, - \, \alpha^{-1} \, | \ \le \ q^{-d}$, we consider three cases for $y$: \   split, elliptic unramified and elliptic ramified.

\medskip
{\fontfamily{phv}\selectfont{Case $y$ split or elliptic unramified:}}  \ \ Here, the eigenvalues of $y$ belong to either $\mF$ or an unramified quadratic extension, and therefore $| \, \alpha \, - \, \alpha^{-1} \, |$ is a (positive) integral power of $\frac{1}{q}$, so  $| \, \alpha \, - \, \alpha^{-1} \, | \ \le \ q^{-d}$ in fact means $| \, \alpha \, - \, \alpha^{-1} \, | \ < \ q^{-d}$.  By the Sally-Shalika character tables [{\reSSa}], if $\pi$ is an irreducible supercuspidal representation of depth $d$:

\begin{equation}\label{ss-ramified-a}
\Theta_{\pi}( \, y \, ) \ = \ 
\begin{cases}  
\ 0 &{\text{\rm{$y$ split or unramified elliptic, and  $| \, \alpha \, - \, \alpha^{-1} \, | \ > \ {q^{-d}}$}}} \\
\ &\ \\
{\frac{1}{| \, \alpha \, - \, \alpha^{-1} \, |}} \, - \, {\frac12} q^{d+{\frac12}} \big( {\frac{q+1}{q}} \big)  &{\text{\rm{when $y$ is split and $| \, \alpha \, - \, \alpha^{-1} \, | \ \le \ {q^{-d}}$}}} \\
\ &\ \\
- \, {\frac12} q^{d+{\frac12}} \big( {\frac{q+1}{q}} \big) &{\text{\rm{$y$ unramified elliptic, and  $| \, \alpha \, - \, \alpha^{-1} \, | \ \le \ {q^{-d}}$}}} \\
\end{cases}
\end{equation}

\medskip

\noindent{So},

$$
\aligned
{\frac{\meas ({\mathcal K})}{(q-1)(q+1)}} \, e_{d} (y) \ &= \ 
\begin{cases} 
\ 0 {\hskip 1.07in} {\text{\rm{$y$ split or unramified elliptic, and  $| \, \alpha \, - \, \alpha^{-1} \, | \ > \ {q^{-d}}$}}} \\
\ &\ \\
\ {\dfrac{\big( \, 2 (q-1) \ q^{2(d-{\frac12})} \, \big)}{| \, \alpha \, - \, \alpha^{-1} \, |}} \ - \ (q-1)(q+1) \, q^{3(d -{\frac12})} \\
{\hskip 1.2in} {\text{\rm{when $y$ is split and $| \, \alpha - \alpha^{-1} \, | \, \le \, q^{-d}$}}} \\
\ \\
\ (q-1)(q+1)\, q^{3(d -{\frac12})} \, (-1) \\ 
{\hskip 1.2in} {\text{\rm{$y$ unramified elliptic and $| \, \alpha - \alpha^{-1} \, | \, \le \, q^{-d}$}}} \\
\end{cases}
\endaligned
$$

\bigskip
{\fontfamily{phv}\selectfont{Case $y$ ramified:}} \ \ As already mentioned, by the support condition, we may and do assume $| \, \alpha - \alpha^{-1} \, | \, \le \, q^{-d}$.  \ We consider two subcases depending on whether: 
$$
| \, \alpha - \alpha^{-1} \, | \, < \, q^{-d} \qquad {\text{\rm{or}}} \qquad  | \, \alpha - \alpha^{-1} \, | \, = \, q^{-d} \ .
$$
\medskip

{\fontfamily{phv}\selectfont{Subcase \ $| \, \alpha - \alpha^{-1} \, | \, = \, q^{-d}$}}:    \ The element $y$ has a conjugate in the set ${\mathcal I}_{d} \ \backslash \ {\mathcal I}_{d+\frac12}$.  We may and do assume $y$ is in ${\mathcal I}_{d} \ \backslash \ {\mathcal I}_{d+\frac12}$.  We remark that under the isomorphism of  ${\mathcal I}_{d}/{\mathcal I}_{d+\frac12}$ with 
${\fkg_{x_{01},_{d}}}/{\fkg_{x_{01},_{d+\frac12}}}$, the coset $y$ is a non-degenerate coset.   

\begin{lemma} Let $\psi$ be a non-trivial character of $\myFF$.  Suppose $u, \, v \ \in \ (\myFF)^{\times}$.    Then, 
$$ 
{\underset {a, \, b \ \in \ (\myFF)^{\times}} \sum} \psi \, (\, u \, a \, ) \ \psi \, (\, v \, b \, ) \ = \ 1 \ . 
$$
\end{lemma}
\begin{proof} By changes of variables 
$$
{\underset {a, \, b \ \in \ (\myFF)^{\times}} \sum} \psi \, (\, u \, a \, ) \ \psi \, (\, v \, b \, ) \ = \ \Big( \,  {\underset {c \ \in \ (\myFF)^{\times}} \sum} \psi \, ( \, c \, ) \, \Big)^2 \ = \ (-1)^2 \ = \ 1 \ .
$$
\end{proof}
\begin{cor} If $y \in {\mathcal I}_{d} \ \backslash \ {\mathcal I}_{d+\frac12}$ is a ramified elliptic element whose eigenvalues $\alpha, \ \alpha^{-1}$ satisfy 
$| \, \alpha - \alpha^{-1} \, | \, = \, q^{-d}$, then 
$$
\tau \, ( \, y \, ) \ = \ 1 \ .
$$
\end{cor}

Let ${\overset {\text{\rm{\bf{.}}}} {\tau} }$ denote the extension of $\tau $ to a function on $\mG$ which is zero outside ${\mathcal I}_d$.  Harish-Chandra's formula [{\reHCa}] for the character $\Theta_{\tau}$ is: 
$$
\Theta_{\tau}(x) \ = \ {\underset {g{\mathcal I}_d \, \in \, \mG / {\mathcal I}_d } \sum } \ {\overset {\text{\rm{\bf{.}}}} {\tau} } \, ( \, g \, x \, g^{-1} \, ) \ .
$$ 
\noindent{Since} $y \in {\mathcal I}_d$ has the property that under the isomorphism ${\mathcal I}_{d}/{\mathcal I}_{d+\frac12} \, \simeq \, {\fkg_{x_{01},_{d}}}/{\fkg_{x_{01},_{d+\frac12}}}$ it corresponds to a non-degenerate coset, the only $g{\mathcal I}_d$ satisfying $g \, y \, g^{-1} \, \in \, {\mathcal I}_d$ is when $g \in {\mathcal I}$.   Whence, 
$$
\aligned
\Theta_{\tau}(y) \ &= \ {\underset {g{\mathcal I}_d \, \in \, \mG / {\mathcal I}_d } \sum } \ {\overset {\text{\rm{\bf{.}}}} {\tau} } \, ( \, g \, y \, g^{-1} \, ) \ = \ {\underset {g{\mathcal I}_d \, \in \, {\mathcal I} / {\mathcal I}_d } \sum } \ {\overset {\text{\rm{\bf{.}}}} {\tau} } \, ( \, g \, y \, g^{-1} \, ) \\
&\ \\
&= \ [ {\mathcal I} :  {\mathcal I}_d ] \, {\overset {\text{\rm{\bf{.}}}} {\tau}}  (y)  \ = \ (q-1) q^{3(d-\frac12 )} {\tau}  (y) \ = \ (q-1) q^{3(d-\frac12 )} \ ;
\endaligned
$$ 
\noindent{and}
$$
{\frac{\meas ({\mathcal K})}{(q-1)(q+1)}} \, e_{d} (y) \ = \ q^{3(d - \frac12 )} \qquad {\text{\rm{when}}} \qquad  | \, \alpha - \alpha^{-1} \, | \, = \, q^{-d} \ .
$$

\medskip

\medskip
{\fontfamily{phv}\selectfont{Subcase \ $| \, \alpha - \alpha^{-1} \, | \, < \, q^{-d}$}}:  \  We note irreducible supercuspidal representations of depth $d$ come in pairs $\pi$ and $\pi'$, i.e., an $L$-packet.  In the Sally-Shalika [{\reSSa}] parameterization of ramified irreducible supercuspidal representation, each element of the pair corresponds to taking one of two classes of additive characters of $F$, and their character table gives: 
$$
\aligned
\big( \, \Theta_{\pi} \ + \ \Theta_{\pi'} \, \big) (y) \ &= \ - \, q^{d+{\frac12}} \ {\dfrac{(q+1)}{q}} \ = \ - \, q^{d-{\frac12}} \, (q+1) \qquad {\text{\rm{when \ \ $| \, 1 - \alpha \, | \, < \, q^{-d}$}}} \ ;
\endaligned
$$
\noindent{so},
$$
\aligned
( \, e_{\pi} \ + \ e_{\pi'} \, ) \, (y) \ &= \ {\dfrac{(q-1)(q+1)q^{(d- \frac12 )}}{2 \, {\meas ({\mathcal K})} }} \cdot \big( \, - \, q^{d-{\frac12}}(q+1) \, \big) \ . \\
\endaligned
$$
\noindent{Whence},
$$
\aligned
e_{d} \, (y) \ &= \ 2 \, (q-1) \, q^{(d - \frac12 )} \ \Big( \ {\dfrac{(q-1)(q+1)q^{(d- \frac12 )}}{2 \, \meas ({\mathcal K}) }}  \ \Big) \cdot \big( \, - \, q^{d-{\frac12}}(q+1) \, \big) \\
&= \ - \, {\dfrac{(q-1)^2 \, (q+1)^2}{ \meas ({\mathcal K}) }} \, q^{3(d-{\frac12})} \ , \\
\endaligned
$$
\noindent{i.e.,}
$$
\aligned
{\frac{ \meas ({\mathcal K}) }{(q-1)(q+1)}} \, e_{d} \, (y) \ &= \ (q-1) \, (q+1) \ q^{3(d-{\frac12})} \ (-1) \ . \\
\endaligned
$$

\bigskip

\vfill

\begin{center}
\begin{tabular}{ | c | l | }
    \hline 
\multicolumn{2}{| c |}{ {\fontfamily{phv}\selectfont{Table 3}} $\begin{matrix} \ \\ \ \end{matrix}$ \quad $d \in {\mathbb N}$ \ \ half-integral } \\
    \hline 
\multicolumn{2}{| c |}{$y$ has eigenvalues $\alpha$, $\alpha^{-1}$ \ : \qquad value of \ ${\dfrac{1}{(q-1)(q+1)}} \ e_{d}(y) \begin{matrix} \ \\ \ \\ \ \end{matrix}$}   \\
    \hline
    \hline
\ &\ \\
{$y$ split} &$\begin{cases}
\ 0  {\hskip 0.50in} {\text{\rm{ when $\alpha \ \notin 1+{\wp}^{d+{\frac12}}_{F}$, i.e., \  $| \, 1 \, - \, \alpha \, | \ > \ {\frac{1}{q^{d}}}$}}} \\
\ \\ 
{\dfrac{\big( \, 2 (q-1) \ q^{2(d-{\frac12})} \, \big)}{| \, \alpha \, - \, \alpha^{-1} \, |}} \ - \ (q-1)(q+1) \, q^{3(d -{\frac12})} \ \ \ \ {\text{\rm{when $| \, \alpha - \alpha^{-1} \, | \, \le \, q^{-d}$}}} \\
\end{cases}$
 \\
\ &\ \\
    \hline
\ &\ \\
{${\begin{matrix} 
{\text{\rm{$y$ ramified}}} \\
{\text{\rm{elliptic}}} \end{matrix}}$} 
&$\begin{cases}
\ 0  {\hskip 1.28in} {\text{\rm{ when $| \, 1 \, - \, \alpha \, | \ > \ {\frac{1}{q^{d}}}$}}} \\
\ \\
\ q^{3(d-{\frac12})} \  {\hskip 0.90in} {\text{\rm{for \ $| \, 1 \, - \, \alpha \, |_{E} \ = \ {\frac{1}{q^{d}}}$  }}} \\
\ \\
\ (q-1)\, (q+1) \ q^{3(d-{\frac12})} \ (-1)  {\hskip 0.30in} {\text{\rm{for \ $| \, 1 \, - \, \alpha \, |_{E} \ < \ {\frac{1}{q^{d}}}$  }}} \\
\end{cases}$
 \\
\ &\ \\
    \hline
\ &\ \\
{${\begin{matrix} 
{\text{\rm{$y$ unramified}}} \\
{\text{\rm{elliptic}}} \end{matrix}}$} 
&$\begin{cases}
\ 0  {\hskip 0.50in} {\text{\rm{ when  $| \, 1 \, - \, \alpha \, | \ > \ {\frac{1}{q^{d}}}$}}} \\
\ \\ 
\ (q-1) \ (q+1) \ q^{3(d-{\frac12})} \ (-1)  {\hskip 0.30in} {\text{\rm{for \ $| \, 1 \, - \, \alpha \, | \ < \ {\frac{1}{q^{d}}}$ }}} \\
\end{cases}$ \\
\ &\ \\
    \hline
\end{tabular} {\hfil} 
\end{center}

\vskip 1.0in



\section{The main result} 

For convenience  in numbering, for $k \, \in \, {\frac12}{\mathbb N}$, set 

\begin{equation}
\mytotal_{k} \ := \ e_{0} \ + \ e_{\frac12} \ + \ \cdots \ + \ e_{k} \ . 
\end{equation}

\noindent{We} note in particular $\mytotal_{0} \,  = \, e_{0}$.
 
\bigskip

\begin{thm}\label{main-result-group} \ \  For $k \, \in \, {\frac12}{\mathbb N}$, set $k^{+} \, := \, k \, + \, {\frac12}$.  Under the assumption the p-adic field $\mF$ has odd residue characteristic, we have \ ${\text{\rm{supp}}} \, ( \,  {\mytotal}_{k} \, )  \ \subset \ {\mathcal U}^{\text{\rm{top}}}_{k^{+}}$.  \ On \ ${\mathcal U}^{\text{\rm{top}}}_{k^{+}}$:

\medskip

\begin{itemize}


{\mycolor

\item[$\bullet$] \ \ When  $k$ is integral:

$$
\mytotal_{k} \, ( \, y \, ) \, = \, 
{\text{\rm{\Large $(q^2-1) \, q^{3k} $}}} \ 
\begin{cases}
\  \big( {\frac{2 \, q^{-k}}{| \, \alpha \, - \, \alpha^{-1} \, |_{\mF}}} \, - \, 1 \big) \\ 
\qquad {\text{\rm{when $y$ is split with eigenvalues $\alpha$, $\alpha^{-1}$}}} \\
\ &\ \\
\ -1  \\
\qquad {\text{\rm{when $y$ is elliptic}}} \\
\end{cases}
$$

\medskip

\item[$\bullet$] \ \ When $k$ is half-integral:

$$
\mytotal_{k} \, ( \, y \, ) \, = \, 
{\text{\rm{\Large $(q^2-1) \, q^{3k+{\frac12}} $}}} \ 
\begin{cases}
\  \big( {\frac{2 \, q^{-(k+{\frac12})}}{| \, \alpha \, - \, \alpha^{-1} \, |_{\mF}}} \, - \, 1 \big) \\ 
\qquad {\text{\rm{when $y$ is split with eigenvalues $\alpha$, $\alpha^{-1}$}}} \\
\ &\ \\
\ -1  \\
\qquad {\text{\rm{when $y$ is elliptic}}} \\
\end{cases}
$$
}

\end{itemize}

\end{thm}

\medskip

\begin{proof} \ The proof is induction on the depth:  \ The values of $\mytotal_{0} \, = \, e_{0}$ are given in Table 1.  Given $\mytotal_{k}$, we compute $\mytotal_{k^{+}} \, = \, \mytotal_{k} \, + \, e_{k^{+}}$ via Table 3 when $k$ is integral, and via \eqref{ps-greater-zero} and Table 2 when $k$ is half-integral.  
\end{proof}

For $k \, \in \, {\frac12}\bN$, we note the series defining the exponential and logarithm maps between $\fkg$ and $\mG$ converge for $k$ sufficiently large.  We take $k_{0} \, \in \, {\frac12}\bN$ so that:

\begin{equation}\label{lie-algebra-condition}
{\text{\rm{ $\myexp$ and $\mylog$ are bijections between $\fkg_{k}$ and $\mG_{k}$ when $k \, \ge \, k_{o}$.}}}
\end{equation}

\begin{cor} Under condition \eqref{lie-algebra-condition}{\,}, so that $\sigma_{k} \circ \myexp$ is defined and has support in ${\mathcal N}^{\text{\rm{top}}}_{k^{+}}$, if $Y \, \in \, {\mathcal N}^{\text{\rm{top}}}_{k^{+}}$ has eigenvalues $\pm \, \lambda$, then :

\medskip

\begin{itemize}


{\mycolor

\item[$\bullet$] \ \ When  $k$ is integral and $Y \, \in \, \fkg_{k^{+}}$:

$$
\mytotal_{k} \, \circ \, \myexp ( \, Y \, ) \, = \, 
{\text{\rm{\Large $(q^2-1) \, q^{3k} $}}} \ 
\begin{cases}
\  \big( {\frac{2 \, q^{-k}}{| \, \lambda |_{\mF}}} \, - \, 1 \big) \\ 
\qquad {\text{\rm{when $y$ is split}}} \\
\ &\ \\
\ -1  \\
\qquad {\text{\rm{when $y$ is elliptic}}} \\
\end{cases}
$$

\medskip

\item[$\bullet$] \ \ When $k$ is half-integral and $Y \, \in \, \fkg_{k^{+}}$:

$$
\mytotal_{k}  \, \circ \, \myexp ( \, Y \, ) \, = \, 
{\text{\rm{\Large $(q^2-1) \, q^{3k+{\frac12}} $}}} \ 
\begin{cases}
\  \big( {\frac{2 \, q^{-(k+{\frac12})}}{| \, \lambda \, |_{\mF}}} \, - \, 1 \big) \\ 
\qquad {\text{\rm{when $y$ is split}}} \\
\ &\ \\
\ -1  \\
\qquad {\text{\rm{when $y$ is elliptic}}} \\
\end{cases}
$$
}

\end{itemize}

\noindent{In particular}, $\mytotal_{k+1}  \, \circ \, \myexp$ and $\mytotal_{k}  \, \circ \, \myexp$ satisfy the homogeneity relation:
 
\begin{equation}\label{homogeneity-relation-a}
( \, \mytotal_{k+1}  \, \circ \, \myexp \, ) \, ( \, \varpi Y \, ) \  = \  q^{3} \,  ( \, \mytotal_{k}  \, \circ \, \myexp \, ) \, ( \, Y \, ) \ .
\end{equation}

\end{cor}

\medskip

Remarks: \ \ (i) \ Examination of the explicit formulae for $\mytotal_{k}$ and 
  $\mytotal_{k}  \, \circ \, \myexp$, shows these distributions depend only on the characteristic polynomial of the input, and therefore are stable distributions.

\smallskip

\ (ii) \ The formula for $\sigma_{0} \, = \, \epsilon_{0}$ leads to the observation that it is the restriction of the Steinberg character to ${\mathcal U}^{\text{\rm{top}}}$. 

\smallskip

\ (iii) \ The transfer of $\sigma_{k}$ to $\sigma_{k} \circ \myexp$ is only valid when $k \, \ge \, k_{0}$, but the homogeneity relation \eqref{homogeneity-relation-a} allows us to formally continue $\sigma_{k} \circ \myexp$ to the range $0 \, \le \, k \, < \, k_{0}$, so the continuation for parameter has support in $\fkg_{k^{+}}$, e.g., $\fkg_{0^{+}} \, = \, {\mathcal N}^{\text{\rm{top}}}$ when $k \, = \, 0$.  This is analogous to the behavior of $\sigma_{k}$'s in the same range.  

\smallskip

\ (iv) \ The power $3$ of the factor $q^3$ should be viewed as the dimension of $\fkg$.  Under Fourier transform on the Lie algebra, the homogeneity relation \eqref{homogeneity-relation-a} becomes a 

\begin{equation}\label{homogeneity-relation-b}
\myFT \, ( \, \mytotal_{k+1}  \, \circ \, \myexp \, ) \, ( \, \varpi^{-1} Y \, ) \  = \ \myFT \,  ( \, \mytotal_{k}  \, \circ \, \myexp \, ) \, ( \, Y \, ) \ .
\end{equation}

\bigskip

To identify $\myFT \,  ( \, \mytotal_{k}  \, \circ \, \myexp \, )$, we note the following Proposition, whose proof is in the appendix:

\begin{prop*}{\bf \ref{appendix-prop}}   \ \ For $\fkg \, = \, \fks \fkl (2, \mF)$, we have

\vskip 0.10in

\begin{itemize} 
\item[$\bullet$] The Fourier transforms $FT(1_{\fkg_{0}} )$ and $FT(1_{\fkg_{-{\frac12}}} )$ have support in the sets $\fkg_{0^{+}} := \fkg_{\frac12}$  and 
$\fkg_{ ({\frac12})^{+}} := \fkg_{1}$ respectively.  \ \  In particular, the support is contained in ${\mathcal N}^{\text{\rm{top}}}$. 

\vskip 0.10in

\item[$\bullet$] For $k \ge 1$, the Fourier transform
$FT(1_{\fkg_{-k}} )$ has support in $\fkg_{k^{+}} := \fkg_{k+\frac12}$.
\end{itemize}
\end{prop*}

For a general connected reductive p-adic group, under conditions in which the exponential map takes ${\mathcal N}^{\text{\rm{top}}}_{r}$ to ${\mathcal U}^{\text{\rm{top}}}_{r}$ ($r \, > \, 0$), Kim [{\reKa},{\reKb}], showed, for $X$ in $\fkg_{\text{\normalsize{\rm{$({\frac{d}{2}})^{+}$}}}}$:

$$
\int_{\widehat{G}^{{\text{\tiny{\rm{{\,}temp}}}}}_{\le d}} \ \Theta_{\pi} ( \, \exp (X) \, ) \, d \mu_{_{\text{PM}}} \, ( \pi ) \ = \ 
FT(1_{\fkg_{-d}}) \, (X) \ ,
$$

\noindent{where} ${\widehat{G}^{{\text{\tiny{\rm{{\,}temp}}}}}_{\le d}}$ is the (classes of) irreducible tempered representations of depth less than or equal to $d$.  In this situation, for $\SL (2)$, we have 
$$
\sigma_{d} \circ \exp  \ = \ FT(1_{\fkg_{-d}}) \qquad {\text{\rm{(both sides have support in $\fkg_{d^{+}}$).}}}
$$

\noindent{We} conjecture, for $\SL(2)$, and more generally for any connected reductive p-adic group, this identity is true when the depth is sufficiently large.

\vfill

\eject

\appendix
\section{\ }

\medskip
Here $\mG \, = \, \SL (2, \mF)$, and $\fkg \, = \, \fks \fkl (2, \mF)$.  Suppose $\psi$ is an additive character of $\mF$ with conductor $\wp$.  Let $\myFT$ denote the Fourier transform on $\fkg$, i.e., if $f \in C^{\infty}_{c}(\fkg )$:

$$
\myFT \, (f) (Y) \ = \ \int_{\fkg} \psi \, ( \, \mytrace (\, X \, Y \, ) \, ) \ f(X) \ dX \ .
$$ 

\noindent{In} this appendix we prove:

\begin{prop}\label{appendix-prop}  \ \ For $\fkg \, = \, \fks \fkl (2, \mF)$, we have

\vskip 0.10in

\begin{itemize} 
\item[$\bullet$] The Fourier transforms $FT(1_{\fkg_{0}} )$ and $FT(1_{\fkg_{-{\frac12}}} )$ have support in the sets $\fkg_{0^{+}} := \fkg_{\frac12}$  and 
$\fkg_{ ({\frac12})^{+}} := \fkg_{1}$ respectively.  \ \  In particular, the support is contained in ${\mathcal N}^{\text{\rm{top}}}$. 

\vskip 0.10in

\item[$\bullet$] For $k \ge 1$, the Fourier transform
$FT(1_{\fkg_{-k}} )$ has support in $\fkg_{k^{+}} := \fkg_{k+\frac12}$.
\end{itemize}
\end{prop}

\begin{proof}  Since $1_{\fkg_{0}}$, and $1_{\fkg_{-{\frac12}}}$ are $\myAd \, (\mG )$-invariant sets, their Fourier transforms are $\myAd \, (\mG )$-invariant.  Therefore, it is sufficient to show the stated vanishing on any convenient element in an 
Adjoint orbit.   

\smallskip

We prove the result for $1_{\fkg_{0}}$ and remark our argument proof is easily adapted to also treat the case $1_{\fkg_{-\frac12}}$.  We have 

$$
\myFT \, (1_{\fkg_{0}} ) \, ( \, Y \, ) \ = \ \myPV \, \int_{\fkg}  \ \psi ( \, \mytrace ( \, X \, Y \, ) \ 1_{\fkg_{0}}( \, X \, ) \ dX \quad ({\text{\rm{principal value}}}) .
$$
\noindent{We} note that $\fkg_{0} \, = \, \{ \, X \in \fkg \ | \ \det (X) \in {\mathcal R}_{\mF} \, \}$.  Let $X \, = \, \left[ \begin{matrix} a &b \\ c &-a \end{matrix} \right]$,  and for integral $\ell$, set 
$$
\aligned
{\mathcal T}_{\ell} \ = \ \{ \ X \in \fkg_{0} \ | \ a, \, b, \, c \ \in \wp^{-\ell} \ \} \ = \ \{ \ X \in \fkg \ | \ \det(X) \in {\mathcal R}_{\mF} \ {\text{\rm{and}}} \ a, \, b, \, c \ \in \wp^{\ell} \ \} \ .
\endaligned
$$

\noindent{We} show for $Y \not\in \fkg_{\frac12}$ the integral
$$
\int_{{\mathcal T}_{\ell}}  \ \psi ( \, \mytrace ( \, X \, Y \, ) \, ) \ 1_{\fkg_{0}}( \, X \, ) \ dX \ = \ \int_{{\mathcal T}_{\ell}}  \ \psi ( \, \mytrace ( \, X \, Y \, ) \, ) \ dX \quad {\text{\rm{vanishes for ${\ell}$ large.}}} 
$$
\smallskip

The Fourier transforms $\myFT ( 1_{\fkg_{j}} )$ are invariant under the Adjoint action of $\GL (2,\mF  )$.  The  $\GL (2,\mF  )$-orbit of a regular semisimple element $Y$ contains an element of the form 
$$
\left[ \begin{matrix} 0 &B \\ C &0 \end{matrix} \right] \quad {\text{\rm{with $\val (B) \, \le \, \val (C) \, \le \, \val (B)+1$.}}} 
$$
We take $Y$ to have this anti-diagonal form.  Here, $Y$ is not topologically nilpotent when $\val (C) \, \le \, 0$.   It thus suffices to  prove $\myFT ( 1_{\fkg_{0}} ) (Y) \, = \, 0$ in this situation.  For $X \, = \, \left[ \begin{matrix} a &b \\ c &-a \end{matrix} \right]$, we have \ $\mytrace ( \, X Y \, ) \, = \, (bC \, + \, cB)$, \ and so

$$
\int_{{\mathcal T}_{\ell}}  \ \psi ( \, \mytrace ( \, X \, Y \, ) \, ) \ dX \ = \ 
\int_{{\mathcal T}_{\ell}}  \ \psi ( \, bC \, + \, cB \, ) \ dX \ .
$$

\bigskip

\noindent {\sc Case \  $\val (C) \, = \, \val (B)$:} \ \ We remark in this situation, $Y$ is either split ore elliptic unramified.  {We} show the integral vanishes if $B, \, C \, \not\in \, \wp$.   Our strategy is to partition ${\mathcal T}_{\ell}$ into regions were the integral is zero. 

\medskip

\noindent{\sc Subcase $a \in {\mathcal R}_{\mF}${\,}:} \ \ The condition for $X$ to be in $\fkg_{0}$ is $a^2 +bc \in {\mathcal R}_{\mF}$, and so $bc \in {\mathcal R}_{\mF}$.  We consider two subcases based on $b$. 

\begin{itemize}

\item[$\bullet$] Subcase $b \in {\mathcal R}$.  \ \  Here, dependent on the valuation $\myval (b)$, the variable $c$ runs over an ideal between $\wp^{0}$ and $\wp^{-\ell}$.  The assumption $B \not\in \wp$ means \ $c \, \rightarrow \, \psi (cB )$ \ is a non-trivial character on its allowed ideal and therefore, for fixed $a$ and $b$, the integral over $c$ is zero.  We deduce the integral over the region in ${\mathcal T}_{\ell}$ which satisfies $a, b \in {\mathcal R}$ is zero.

\medskip

\item[$\bullet$] Subcase $b = \varpi^{-k} u$ with (integral) $k>0$ and $u$ a unit. \ \   The condition $bc \in {\mathcal R}$, is $c \in \wp^{k} \, \subset \, \wp$.  If $\val (B) \le -k$, then $c \, \rightarrow \, \psi (cB)$ is a non-trivial character; so integration of $c$ over $\wp^{k}$ is zero.  If $-k < \val (B)$, then $\psi (cB) \, = \, 1$; so $\psi(bC+cB) \, = \, \psi (bC)$.  Integration over $c \in \wp^{k}$ yields $\psi (bC) \, \meas ( \wp^{k})$.   We note that $x \rightarrow \psi (xC)$ is a non-trivial character on ${\mathcal R}$, and $\psi ((b+x)C) = \psi (bC) \psi (xC)$.  If we integrate over all $b \in \wp^{-k} \backslash  \wp^{-k+1}$ we get zero.  We deduce the integral over the region in ${\mathcal T}_{\ell}$ which satisfies $a \in {\mathcal R}$ and $b \not\in {\mathcal R}$ is zero. 

\smallskip

\end{itemize}

\medskip

\noindent{\sc Case $a \not\in {\mathcal R}_{\mF}${\,}:} \ \ Write $a$ as $a \, = \, \varpi^{-k}u$ with $u$ a unit, and $k$ a positive integer (note $k \le \, \ell$ to satisfy $a \in \wp^{-\ell}$).  The condition $a^2 \, + \, bc \, \in \, {\mathcal R}$ for $X$ to be in $\fkg_{0}$ is thus $-bc \, \in \, \varpi^{-2k}u^2 \, + \, {\mathcal R}$.  In particular, $b$ and $c$ are non-zero.  Write $b$, as $b = \varpi^{\beta} v_{b}$, with $v_b$ a unit.  The condition, $b, \, c \, \in \wp^{-\ell}$ means $-\ell \, \le \, \beta \, \le \, \ell - 2k$, and similarly for the valuation $\gamma = \val (c)$ of $c$.

\medskip

\begin{itemize}

\item[$\bullet$] \ \ If $\beta \neq \gamma$, then by the symmetry of $b$ and $c$, we assume $\beta \, < \, -k \, < \, \gamma$.  This imposes the condition
$c \, \in \, -\varpi^{\gamma} \, v^{-1}_{b} \, u^2 \, + \, \wp^{-\beta}$.  

\smallskip


\item[$\cdot$] If $\val (B) - \beta \, \le \, 0$, then for fixed $a$ and $b$, the integral over $c$ is zero. 

\smallskip

\item[$\cdot$]  If $\val (B) - \beta \, > \, 0$, then $\psi (cB) \, = \, \psi ( -\varpi^{\gamma} \, v^{-1}_{b} \, u^2 \, B )$ is independent of $c \, \in \, -\varpi^{\gamma} \, v^{-1}_{b} \, u^2 \, + \, \wp^{-\beta}$.  Thus, if we fix $a$ and $b$, and integrate $\psi(bC+cB)$ over $c$, we get 
$$
\psi (b C) \, \psi ( -\varpi^{\gamma} \, v^{-1}_{b} \, u^2 \, B ) \, \meas (\varpi^{(-2k-\beta)} + \wp^{-\beta}) \ .
$$
\noindent{If} we perturb $b$ by $v_{b}x \in \wp^{-\val (B)}$ to $b' \, = \, b+v_{b}x$, so $v_{b'} = v_{b} ( 1 + \varpi^{-\beta} x)$, then  the corresponding $c'$ satisfies $c' \, \in \, -\varpi^{\gamma} v^{-1}_{b} (1 + \varpi^{-\beta} x)^{-1} u^2 \, + \, \wp^{-\beta} \, = \, -\varpi^{\gamma} v^{-1}_{b} u^2 + \varpi^{\gamma-\beta} v^{-1}_{b}x u^2 \, + \,  \wp^{-\beta}$.  We deduce $\psi (c'B) \, = \, \psi ( -\varpi^{\gamma} \, v^{-1}_{b} \, u^2 \, B )$ is independent of $x$.  So, $\psi (b'C \, + \, c'B) \, = \, \psi (bC) \, \psi (xC) \, \psi ( -\varpi^{\gamma} \, v^{-1}_{b} \, u^2 \, B )$, and therefore, if we restrict to $b \in \varpi^{-\beta} {\mathcal R}^{\times}$ integrate over $c'$ followed by integration over $b$, we get zero.

\medskip

\item[$\bullet$] \ \ If $\beta = \gamma = -k$, we have 
$$
X \ = \ \varpi^{-k} \left[ \begin{matrix} u &v_{b} \\ v_{c} &-u \end{matrix} \right ] \quad {\text{\rm{with \ $u, \, v_{b}, \, v_{c}$ units}}},   
$$
\noindent{and} the condition for $X \in \fkg_{0}$ is $u^2 \, + \, v_{b} v_{c} \, \in \, \wp^{2k}$.   We see the product $-v_{b}v_{c}$ must be a square in ${\mathcal R}^{\times}$.  Conversely, if $-v_{b}v_{c}$ is a square, the condition on $u$ is 
$u^2 \, \in \,   -v_{b}v_{c} + \wp^{2k}$.  We fix $v_c$ and multiplicatively perturb $v_{b}$ by $1+\wp$, the integral of 
$\psi ( \varpi^{-k} (v_{c} B + v_{b'} C) )$ over $b' \in \varpi^{-k} v_{b} ( 1+\wp )$, and $a \, = \, \varpi^{-k}u$ with  $u^2 \, \in \,  -v_{b'}v_{c} + \wp^{2k}$.  The integration over $a$ yields a constant independent of $b'$, and then the integration over $b'$ is zero.  So the integral of $\psi ( \varpi^{-k} (v_{c} B + v_{b'} C) )$ over the region satisfying $X \in \fkg_{0}$ and $a, \, b, \, c \, \in \varpi^{-k} {\mathcal R}^{\times}$ is zero.  

\end{itemize}

\noindent {\sc Case \  $\val (C) \, = \, \val (B)+1$:} \ \ The proof here is a minor modification of the case  $\val (C) \, = \, \val (B)$.  We omit the details.

\medskip

\noindent{This} completes the proof that $\myFT (1_{\fkg_{0}})$ has support in $\fkg_{\frac12}$.  

\medskip


\medskip 

The statement about the support of $\myFT (1_{\fkg_{-k}})$ for $k \ge 0$ follows from the elementary property  ${\fkg_{k-1}} \, = \, \varpi^{-1} \fkg_{k}$, and elementary homogeneities of the Fourier transform.

\end{proof}

\vfill
\vfill

}}


\begin{thebibliography}{99999}


\bibitem[B]{\reBna} Bernstein, J., Notes of lectures on Representations of p-adic Groups, Harvard University, Fall 1992,  written by K.~E.~Rumelhart.

\bibitem[BD]{\reBD} Bernstein, J., r\'edig\'e par Deligne, P., 
{\it Le ``centre" de Bernstein}, ``Repr\'esentations des Groupes 
 R\'eductifs sur un Corps Local"  written by J.-N. Bernstein, P. Deligne, D. Kazhdan and M.-F. Vign\'eras, Hermann, Paris, 1984.

\bibitem[BKV]{\reBKV} Bezrukavnikov, R., Kazhdan, D., Varshavsky, V., On the depth r Bernstein projector,  arXiv:1504.01353v1, 36 pages.






\bibitem[D]{\reDa}  Dat, J.-F., Quelques propri{\'{e}}t{\'{e}}s des idempotents centraux des groupes p-adiques. (French) [Some properties of the central idempotents of p-adic groups] J. Reine Angew. Math. 554 (2003), pp. 69--103. 

\bibitem[IR]{\reIR}  Ireland, K., Rosen, M., A classical introduction to modern number theory. Second edition. Graduate Texts in Mathematics, 84. Springer-Verlag, New York, 1990.

\bibitem[HC]{\reHCa}  Harish-Chandra, Harmonic analysis on reductive p-adic groups. Notes by G. van Dijk. Lecture Notes in Mathematics, Vol. 162. Springer-Verlag, Berlin-New York, 1970

\bibitem[Ka]{\reKa}  Kim, J.-L., Dual blobs and Plancherel formulas. Bull. Soc. Math. France 132 (2004), no. 1, 55--80.

\bibitem[Kb]{\reKb}  Kim, J.-L., Supercuspidal representations: an exhaustion theorem. J. Amer. Math. Soc. 20 (2007), no. 2, 273--320 (electronic).


\bibitem[MPa]{\reMPa} Moy, A., Prasad, G., Unrefined minimal K-types for p-adic groups. Invent. Math. 116 (1994), no. 1-3, 393--408.

\bibitem[MPb]{\reMPb} Moy, A., Prasad, G., Jacquet functors and unrefined minimal K-types. Comment. Math. Helv. 71 (1996), no. 1, 98--121.


\bibitem[MTa]{\reMTa} Moy, A., Tadi{\'{c}}, M., The Bernstein center in terms of invariant locally integrable functions. Represent. Theory 6 (2002), pp. 313--329 (electronic).


\bibitem[MTb]{\reMTb} Moy, A., Tadi{\'{c}}, M., Erratum to: "The Bernstein center in terms of invariant locally integrable functions'' Represent. Theory 6 (2002), 313--329 (electronic). Represent. Theory 9 (2005), 455--456 (electronic).


\bibitem[SS]{\reSSa} Sally, P. J., Jr., Shalika, J. A., Characters of the discrete series of representations of SL(2) over a local field. Proc. Nat. Acad. Sci. U.S.A. 61 (1968), pp. 1231--1237. 




\end{thebibliography}
\end{document}